\documentclass[10pt]{article}

\usepackage[dvipdf]{graphicx}
\usepackage[latin1]{inputenc}
\usepackage{latexsym}
\usepackage{amsmath,amssymb,amsfonts,amsthm}
\usepackage{xcolor}
\usepackage{mathrsfs}

\usepackage{enumerate}

\usepackage[numbers,comma,sort]{natbib}

\definecolor{db}{RGB}{0, 0, 130}
\usepackage[colorlinks=true,citecolor=red,linkcolor=db,urlcolor=blue,pdfstartview=FitH]{hyperref}

\usepackage{pdfsync}

\definecolor{rp}{rgb}{0.25, 0, 0.75}
\definecolor{dg}{rgb}{0, 0.5, 0}

\newcommand{\R}{\mathbb{R}}
\newcommand{\ER}{\mathbb{R}}

\newcommand{\N}{\mathbb{N}}
\newcommand{\EE}{\mathbb{E}}
\newcommand{\ES}{\mathbb{E}}
\newcommand{\PP}{\mathbb{P}}
\newcommand{\PE}{\mathbb{P}}
\newcommand{\dd}{~d}   

\newcommand{\bqn}{\begin{equation}}
\newcommand{\eqn}{\end{equation}}
\newcommand{\bqne}{\begin{equation*}}
\newcommand{\eqne}{\end{equation*}}

\newcommand{\C}[1]{\mathbf{C{#1}}}
\newcommand{\kerG}{\mathcal{G}}
\newcommand{\Md}{\mathbb{M}_d}
\newcommand{\puisr}{r}

\makeatletter
\newcommand{\customlabel}[2]{%
   \protected@write \@auxout {}{\string\newlabel {#1}{{#2}{\thepage}{#2}{#1}{}}}%
   \hypertarget{#1}{#2\hspace{-0.14cm}}
}
\makeatother

\newtheorem{theorem}{Theorem}

\newtheorem{definition}{Definition}[section]
\newtheorem{prop}[definition]{Proposition}

\newtheorem{example}[definition]{Example}

\newtheorem{lemma}[definition]{Lemma}

\newtheorem{proposition}[definition]{Proposition}
\newtheorem{remark}[definition]{Remark}

\author{Fabien Panloup\footnote{Laboratoire Angevin de Recherche en Math\'ematiques, Universit\'e d'Angers. E-mail: \texttt{fabien.panloup@math.univ-angers.fr}} \and 
Alexandre Richard\footnote{Universit\'e Paris-Saclay, CentraleSup\'elec, MICS and CNRS FR-3487, France.  E-mail: \texttt{alexandre.richard@centralesupelec.fr}.}}

\title{ \Large{\textbf{Sub-exponential convergence to equilibrium for Gaussian driven Stochastic Differential Equations with semi-contractive drift}}}
\begin{document}

\maketitle

\begin{abstract}
The convergence to the stationary regime is studied for Stochastic Differential Equations driven by an additive Gaussian noise and evolving in a \textit{semi-contractive} environment, $i.e.$ when the drift is only contractive out of a compact set but does not have \textit{repulsive} regions. In this setting, we develop a \textit{synchronous} coupling strategy to obtain sub-exponential bounds on the rate of convergence to equilibrium in Wasserstein distance. Then by a \textit{coalescent} coupling close to terminal time, we derive a similar bound in total variation distance.
\end{abstract}

\section{Introduction}

We consider a class of non-Markovian Stochastic Differential Equations (SDE) in $\R^d, d\ge 1$, driven by a Gaussian process $(G_t)_{t\ge0}$ with stationary increments and independent coordinates, of the following form:
\begin{equation}\label{eds}
dX_t=b(X_t) dt+\sigma dG_t
\end{equation}
where $b:\R^d\rightarrow\R^d$ is a (at least) continuous vector field and $\sigma$ a $d\times d$ constant invertible matrix. \smallskip

In a seminal paper, Hairer~\cite{hairer} provided a Markovian structure above such SDEs driven by fractional Brownian motion (fBm) with the help of the following Mandelbrot-Van Ness decomposition of the fBm of Hurst parameter $H\in(0,1)$:
\begin{align}\label{eq:movavfbm}
\forall t\in \R,\quad B^H_t = c_H \int_\R \left\{(t-u)_+^{H-\frac{1}{2}} - (-u)_+^{H-\frac{1}{2}}\right\} ~dW_u .
\end{align}
A series of ergodicity results on the existence and uniqueness of the invariant distribution were then established, including rates of convergence to equilibrium in total variation distance. A significant stream of literature followed, focusing in particular on: the elaboration of an ergodic theory for SDEs with extrinsic memory \cite{hairer2,Varvenne}, extensions to SDEs with multiplicative fractional noise \cite{hairer-pillai,fontbona-panloup,DPT}, etc. Moreover, the recent developments in statistical estimation for fractional SDEs \cite{NeuenkirchTindel,ComteMarie,HuNualartZhou} benefited from this theory.

The strategy in \cite{hairer} is to develop a \textit{coalescent coupling} method in this setting. We recall that coalescent coupling means that one tries to stick together two coupled paths $X$ and $Y$ of the SDE starting from different initial conditions. However, since the  SDE is not Markovian, or more precisely since the increments of the fBm depend on the whole past, a non-trivial coupling is needed for the paths $X$ and $Y$ to remain together after being sticked. Unfortunately,  getting the paths together generates a waiting time between the coupling attempts which is very large and this constraint leads to very slow rates of convergence of order $ t^{-(\alpha-\varepsilon)}$ for any $\varepsilon>0$,  
where 
\begin{equation*}
\alpha=\begin{cases}\frac{1}{8}&\textnormal{if $H\in (\frac{1}{4},1)\backslash\left\{\frac{1}{2}\right\}$}\\
 H(1-2H)&\textnormal{if $H\in(0,\frac{1}{4}]$.}
\end{cases}
\end{equation*}
In particular, this is very far from the (usual) exponential rates of the Markovian setting (see e.g. \cite{MeynTweedie}) and it is natural to wonder if other strategies could lead to better rates. 
Before going any further, let us point out that in \cite{hairer}, the drift is roughly assumed to be contractive out of a compact set (this corresponds to Assumption \eqref{C1contract} below).
In the case where $b$ derives from a potential, $i.e.$ $b=-\nabla U$ with $U:\ER^d\rightarrow\ER$, this means that $U$ is uniformly strictly convex out of a compact set but can have 
several wells in a compact. In other words, there exist some regions where the drift can be repulsive. Of course, the situation is much simpler if the 
contractivity is true everywhere, $i.e.$ if Assumption \eqref{C1contract} holds with $R=0$. In this case, a very simple argument shows that if $X_t^x$ and $X_t^y$
are two solutions of the SDE built with the same fBm and starting from deterministic $x$ and $y$, $t\mapsto \ES[|X_t^x-X_t^y|^2]$ decreases exponentially. In turn, this implies 
that $(X_t^x)$ converges to equilibrium with an exponential rate. However, this situation is not representative of the general rate of convergence to equilibrium  of fractional SDEs since this property is only a consequence of the contractive effect of the drift term.\smallskip

Here, we propose to consider a situation which is halfway between the setting of \cite{hairer} and the one described above. More precisely, we assume
that the drift is only contractive out of a compact set but does not have any repulsive regions. In the case $b=-\nabla U$, the typical situation is the one where 
$U$ is flat in a compact set and strictly convex out of this compact set (see Assumption \eqref{C1} below for precise meaning). Thus, it remains possible to study the rate of convergence through a synchronous coupling, $i.e$ through  a couple of solutions which are driven by the same noise. More precisely, this ensures that the distance between synchronous solutions to \eqref{eds} is a.s. non-increasing. However, in order to deduce (Wasserstein) bounds for the rate of convergence, the difficulty lies in the estimation of the time spent by the couple in the non (strictly) contractive regions. In this context, we build a strategy which leads to a sub-exponential rate which certainly depends on the memory induced by the covariance function. When the noise process is a fBm with Hurst parameter $H$, we show in particular that under the above assumptions on the drift,
$$d({\cal L}(X_t^x),\bar{\nu})\leq C \exp(-\tfrac{1}{C} t^{\frac{2}{3}(1-H)-\varepsilon}),$$
where $d$ stands for $2$-Wasserstein distance, $(X_t^x)$ stands for a solution to \eqref{eds} starting from $x\in\ER^d$, $\bar{\nu}$ denotes the (first marginal of the) invariant distribution and $\varepsilon$ is an arbitrary positive number. By coalescent coupling just before time $t$, we are able to deduce the same bound when $d$ is the total variation distance.

Our approach extends to SDEs driven by more general Gaussian processes $G$ with stationary increments. As for the fBm, such processes have moving-average representations similar to \eqref{eq:movavfbm}: they can be written as the integral of a deterministic matrix-valued kernel $\mathcal{G}:\R_-\rightarrow \mathbb{M}_d$ against Brownian motion. We formulate a set of conditions \eqref{C2} on the behaviour of $\mathcal{G}$ at $-\infty$ and at $0^-$, and obtain again a sub-exponential rate of convergence which reads, in Wasserstein distance:
$$d({\cal L}(X_t^x),\bar{\nu})\leq C  \exp(-\tfrac{1}{C} t^{\gamma-\varepsilon}),$$
for $\gamma \in(0,1)$ which depends only on the behaviour of $\mathcal{G}$ at $-\infty$. Under an additional Assumption \eqref{C3} on the ``invertibility'' of $\kerG$, we are able to deduce the same result in total variation distance. 
In particular, the interest of such a generalisation is to get a better understanding of the properties of the kernel which influence the rate of convergence, as much as to exhibit a large class of Gaussian processes to which this ergodic theory applies.

The paper is organised as follows: In Section \ref{sec:setting}, we present the assumptions on $b$ and $\mathcal{G}$ and state our main results. A short and general overview of the proofs is given in Section \ref{sec:overview}. The proof for the results in Wasserstein distance (see Theorems \ref{th:maingen} and \ref{th:maingen2}), is presented in Sections \ref{sec:waitingtimes}, \ref{sec:contraction} and \ref{sec:proofTh3}: first a sequence of stopping times is built (one has to wait for the memory of the noise to fade), then contraction in finite time is proven, and finally the explicit rate is computed. The extension to total variation convergence (see Theorems \ref{th:maingenTV} and  \ref{th:maingenTV2}), is proven in Section \ref{sec:WassToTV}. Since we consider the existence of a stationary measure for \eqref{eds} with a Gaussian noise other than fBm, we need some results that complement those from \cite{hairer}. These are presented in Appendix \ref{App:InvDistrib}. For the sake of completeness, we finally provide in Appendix \ref{App:PND} some details on the structure of increment stationary $\R^d$-valued Gaussian processes (a Wold-type decomposition and a moving-average representation).

\section{Setting and Main Results}\label{sec:setting}
\subsection{Notations}\label{subsec:notations}
The usual scalar product on $\ER^d$ is denoted by $\langle\, ,\,\rangle$ and for $x=(x_1,\dots,x_d)\in\R^d$, $|x|$ stands for the Euclidean norm. $\Md$ stands for the space of \emph{diagonal} square matrices of size $d$ endowed with any norm $\|\cdot\|$. 
For some probability measures $\nu$ and $\mu$ on $\ER^d$ and $r\ge1$, we denote by ${\cal W}_r(\nu,\mu)$ the $r$-Wasserstein distance between $\nu$ and $\mu$, defined by:
$${\cal W}_r(\nu,\mu)=\inf_{(X,Y):~ {\cal L}(X)=\nu,{\cal L}(Y)=\mu}\ES[|X-Y|^r]^{\frac{1}{r}},$$
where, for a random variable $\Upsilon$, ${\cal L}(\Upsilon)$ denotes its probability distribution. 
For a given measurable space $(E,\mathcal{E})$, the total variation norm of a bounded signed measure $m$ is defined by
$$\|m\|_{TV}=\sup_{f\in {\cal B}(E),\|f\|_\infty\le 1}|m(f)|,$$
where ${\cal B}(E)$ denotes the set of measurable functions $f:E\rightarrow\ER^d$, and $m(f):=\int f(x)~m(dx)$.
Denote by ${\cal C}([0,\infty),\ER^d)$ the space of continuous functions from $[0,\infty)$ with values in $\ER^d$. We introduce a Wasserstein-type distance on the space ${\cal P}({\cal C}([0,\infty),\ER^d))$ of probabilities on ${\cal C}([0,\infty),\ER^d)$, defined by: for $P$ and $Q\in  {\cal P}({\cal C}([0,\infty),\ER^d))$, for $r\ge1$,
$${\cal W}_r^\infty(P,Q)=\inf_{(X_.,Y_.):~ {\cal L}(X_.)=P,\,{\cal L}(Y_.)=Q}\ES[\sup_{t\ge0}|X_t-Y_t|^r]^{\frac{1}{r}}.$$
where $X_.=(X_t)_{t\ge0}$ and $Y_.=(Y_t)_{t\ge0}$. For a $T>0$, one will also use the notation $X_{T+.}:=(X_{T+t})_{t\ge0}$. 
Observe that ${\cal W}_r^\infty$  induces a topology on ${\cal P}({\cal C}([0,\infty,\ER^d))$ which is stronger than the usual weak topology induced by the topology of uniform convergence on compact sets. 
The following norms on functional spaces will be encountered: for continuous functions $(d(t))_{t\in [0,1]}$ and $(w(t))_{t\in \R_+}$,
\begin{align*}
\|d\|_{\infty,[0,1]}=\sup_{t\in[0,1]}|d(t)|
\;\textnormal{ and }\; 
\|w\|_{\beta,\infty} = \sup_{s\geq 0} \frac{|w_s|}{(1+s)^{\beta}}.
\end{align*}

We frequently use the letter $C$ to represent a positive real number whose value may change from line to line. The expressions $a\vee b$ and $a \wedge b$ where $a$ and $b$ are any real numbers, stand respectively for the maximum of $a$ and $b$, and its minimum.

\subsection{The fractional case}\label{subsec:fractionalcase}
For the sake of clarity, we choose to focus first on the case where $(G_t)_{t\ge0}$ is a standard $d$-dimensional fractional Brownian motion  with Hurst parameter $H\in(0,1)$.
In this case, a rigorous definition of invariant distribution has been introduced in  \cite{hairer}. More precisely, with the help of the Mandelbrot-Van Ness representation (see Section \ref{subsec:GeneralCase} for background) and a two-sided version of the noise process denoted by $(B_t^H)_{t\in \ER}$, $(X_t, (B_{s+t}^H)_{s\le 0})_{t\ge0}$ can be realised through a Feller transformation $({\cal Q}_t)_{t\ge0}$. In particular, an initial distribution of the dynamical system $(Y,B^H)$ is a distribution $\mu_0$ on $\ER^d\times{\cal W}$, where ${\cal W}$ is an appropriate H\"older space (cf. Appendix \ref{App:InvDistrib}).  Rephrased in more probabilistic terms, an initial distribution is the distribution of  a couple $(Y_0, (B^H_s)_{s\le0})$. For an initial distribution $\mu$, one denotes by ${\cal L}((X_{t}^{\mu})_{t\ge0})$ the distribution on ${\cal C}([0,+\infty),\ER^d)$ of the process starting from $\mu$.
Then, such an initial distribution is called an invariant distribution if it is invariant by the transformation ${\cal Q}_t$ for every $t\ge0$. With a slight abuse of language, one says that two invariant distributions $\nu_1$ and $\nu_2$ are equivalent if ${\cal L}((X_{t}^{\nu_1})_{t\ge0})={\cal L}((X_{t}^{\nu_2})_{t\ge0})$.

\smallskip

As mentioned before, we assume in this paper that the drift term is contractive only out of a compact set but also that there are no ``repulsive'' regions. This corresponds to the first two items of the following assumption:

\begin{minipage}{0.1\linewidth}
(\customlabel{C1}{$\C1$}):
\end{minipage}
\begin{minipage}[t]{0.9\linewidth}
\raggedright
\begin{enumerate}[]
\item(\customlabel{C1norep}{$\C1_{i}$})
$\forall (x,y)\in (\R^d)^2, \quad \langle x-y, b(x)-b(y)\rangle \leq 0$ .
\item(\customlabel{C1contract}{$\C1_{ii}$})  There exist $\kappa,R>0$ such that 
\begin{equation*}
\forall (x,y)\in (\R^d\setminus B(0,R))^2, \quad \langle x-y, b(x)-b(y)\rangle \leq -\kappa |x-y|^2.
\end{equation*}
\item(\customlabel{C1reg}{$\C1_{iii}$}) $b$ is locally Lipschitz with polynomial growth: there exists $C,N>0$ such that
\begin{align*}
\forall x\in\R^d,\quad |b(x)|\leq C(1+|x|^N).
\end{align*}
\end{enumerate}
\end{minipage}

Moreover, we recall that $\sigma$ is always assumed to be invertible in this paper.

\bigskip

Under \eqref{C1}, it can be shown that existence and uniqueness hold for the solution to \eqref{eds} {(despite the fact that $b$ is only locally Lipschitz continuous)} and that \eqref{eds} can be embedded into a Feller Markov process (see Appendix \ref{App:InvDistrib} for details). Furthermore, existence and uniqueness hold for the invariant distribution $\nu$ (see $e.g.$ \cite{hairer}). We denote by ${\cal Q}\nu$ the (unique) distribution of the whole process and by $\bar{\nu}$ the first marginal of $\nu$. Note that if $(Y_t)_{t\ge0}$ denotes a stationary solution to \eqref{eds}, then 
$\bar{\nu}={\cal L}(Y_t)$ for any $t\ge0$.

As mentioned before, when $b=-\nabla U$ where $U:\ER^d\mapsto\ER$ is ${\cal C}^2$, \eqref{C1} is fulfilled when $U$ is convex everywhere, uniformly strictly convex
out a compact subset of $\ER^d$, and with partial derivatives with polynomial growth.\smallskip

\noindent We are now in a position to state our first main result.

\begin{theorem}\label{th:maingen}
Let $q\geq 1$. Let $X$ be a solution to (\ref{eds}) with $G=B^H$ satisfying $\ES[|X_0|^{q+\upsilon}]<+\infty$ for a positive $\upsilon$. Assume \eqref{C1}. 
Denote by $\nu_t$ the law of $X_t$, $t\ge0$.  For any $\varepsilon>0$, there exists $C_1,C_2>0$ such that
\begin{equation}\label{marginalresult}
\forall t>0,\quad \mathcal{W}_q(\nu_t,\bar{\nu}) \leq C_1 e^{-C_2 t^{\gamma}} ,
\end{equation}
where $\gamma = \frac{2}{3}(1-H)-\varepsilon$.
More generally, 
\begin{equation}\label{functionalresult}
{\cal W}_q^\infty({\cal L}(X_{t+.}),{\cal Q}\nu)\leq C_1 e^{-C_2 t^{\gamma}}.
\end{equation}
\end{theorem}

~

\begin{remark}
\begin{itemize}
\item The order of the rate of convergence decreases with $H$, which is reasonable since the memory increases with $H$.
\item\label{rk:marginalToFunctional}
The functional generalisation \eqref{functionalresult} is an obvious consequence of \eqref{marginalresult}. Actually, our proof is based on a synchronous coupling and \eqref{C1} ensures that
if $X$ and $Y$ are two solutions built with the same fractional Brownian motion, $t\mapsto |X_t-Y_t|^q$ is a.s. non-increasing so that if $(Y_t)_{t\ge0}$ is a stationary solution, 
$${\cal W}_q^\infty({\cal L}(X_{t+.}),{\cal Q}\nu)={\cal W}_q^\infty({\cal L}(X_{t+.}),{\cal L}(Y_{t+.}))\le {\cal W}_q^\infty({\cal L}(X_{t}),{\cal L}(Y_{t}))= \mathcal{W}_q(\nu_t,\bar{\nu}).$$
Let us remark that in the Markovian literature, the functional generalisation holds for any Markovian coupling, which explains that this precision is rarely mentioned.

\item In this paper, we emphasize the dependency in $H$ of $\gamma$ and do not try to make explicit the other constants $C_1$ and $C_2$. Nevertheless, one can check that $C_1$ depends affinely on $\ES[|X_0|^{q+\upsilon}]^{\frac{1}{q+\upsilon}}$ (see Subsection \ref{subsec:proof1}).
\end{itemize}
\end{remark}

When convergence holds in Wasserstein distance, a classical method to deduce total variation bounds is to wait sufficiently that the paths get close and then to attempt a coalescent coupling (once only). This strategy can be applied in the fractional case and a suitable calibration of the parameters leads to the following result : 
\begin{theorem}\label{th:maingenTV}
Let the assumptions of Theorem \ref{th:maingen} be in force with $q=1$ and assume that $b$ is Lipschitz continuous when $H>1/2$. Let $\nu_t$ denote the law of $X_t$. Then,

\smallskip
\noindent 
(i) For any $\varepsilon>0$, there exists $C>0$ such that
\begin{equation*}
\|\nu_t-\bar{\nu}\|_{TV} \leq C e^{-\frac{1}{C} t^{\gamma}} ,
\end{equation*}
where $\gamma = \frac{2}{3}(1-H)-\varepsilon$.\\
(ii) More generally,
\begin{equation}\label{eq:functionalTV}
\|{\cal L}(X_{t+.})-{\cal Q}\nu\|_{TV} \leq C e^{-\frac{1}{C} t^{\gamma}} .
\end{equation}
\end{theorem}
The proof of this theorem is achieved in Subsection \ref{proof:theomaingentv}.
\begin{remark} One is thus able to preserve the orders obtained in Wasserstein distance. This property can be interpreted as follows: the cost for sticking the paths is negligible with respect to  the one that is needed to get the paths close.
\end{remark}
\begin{remark}  It is worth noting that here, the functional result is not a trivial consequence of the marginal one. Actually, \eqref{eq:functionalTV} requires to prove
that the paths can remain stuck between $t$ and $+\infty$ and more precisely, to show that the cost of this request is small enough.
\end{remark}

\subsection{The general case}\label{subsec:GeneralCase}

Consider now the general case where the driving process is a \emph{purely nondeterministic} $\R^d$-valued Gaussian process $G = (G^{(1)},\dots,G^{(d)})$ with stationary increments {(see definition \ref{def:pnd})}. We recall that we also assume that the components of $G$ are independent (see Remark \ref{rk:indepComp1} for a discussion about extension to dependent coordinates).\\
{In this context}, the purely nondeterministic property implies that each component admits a moving-average representation of the following type {(Proposition \ref{prop:MovAvRep})}: there exists an $\Md$-valued function $(\kerG(t))_{t\in\R} = (g_{i}(t))_{t\in\R, i\in\{1,\dots,d\}}$ such that $\forall t>0,~ \kerG(t) =0$, and
\begin{equation}\label{eq:def_noise}
\forall t\in\R_+,\quad G_t = \int_{-\infty}^0 \kerG(u) \left(d W_{t+u} - d W_u\right) ,
\end{equation}
where $W$ is a standard two-sided $\R^d$-valued Brownian motion and $\kerG$ satisfies
\begin{align}\label{eq:kerG}
\forall t\in\R_+,\quad \int_{\R} \|\kerG(u-t) - \kerG(u)\|^2 du <\infty .
\end{align}

\begin{remark} The ``purely nondeterministic'' property is usually defined in a slightly more general way for non-Gaussian processes, but we show in Appendix \ref{App:PND} that it is equivalent to the above moving-average representation in the Gaussian case. This explains our slight abuse of language.\\
 In fact, this assumption means that there is no time-dependent deterministic drift in the noise process (in a sense made precise in  Appendix \ref{App:PND}). We could have taken into account such a drift assuming some growth and regularity conditions, but at the cost of heavier notations, thus we chose not to.
\end{remark} 

In order to be able to extend our main results to the general case, we introduce a set of assumptions on the kernel $\kerG$. 

\vspace{0.1cm}

\begin{minipage}{0.1\linewidth}
(\customlabel{C2}{$\C2$}):
\end{minipage}
\begin{minipage}[t]{0.9\linewidth}
Let $G$ be a purely nondeterministic $\R^d$-valued Gaussian process with stationary increments, with kernel $\kerG$ (which thus satisfies \eqref{eq:def_noise} and \eqref{eq:kerG}). Assume further that:
\raggedright
\begin{enumerate}[]
\item(\customlabel{C2supp}{$\C2_{i}$}) $\kerG$ belongs to $\mathcal{C}^2(\R_-^*;\Md)$ and $\kerG$ is not uniformly zero;
\item(\customlabel{C2reg}{$\C2_{ii}$}) There exist $C_1>0$ and $\alpha>-\frac{1}{2}$ such that
\begin{align*}
\forall u\leq-1,\quad \|\kerG''(u)\|\leq C_1 (-u)^{-\alpha-2};
\end{align*}
\item(\customlabel{C2Holder}{$\C2_{iii}$}) There exist $\zeta<\tfrac{1}{2} -(\alpha)_-$ and $C_2>0$ such that 
\begin{align*}
\forall u\in[-2,0),\quad \|\kerG''(u)\|\leq C_2 (-u)^{-\zeta-2} .
\end{align*}
\end{enumerate}
\end{minipage}

\vspace{0.3cm}
\begin{example}
$\bullet$ An $\R^d$-fractional Brownian motion (with possibly different values of Hurst parameter on each component) satisfies \eqref{C2}. In that case, $g_i(u)$ is of the form $(-u)_+^{H-\frac{1}{2}}$, $H\in(0,1)$.\\
$\bullet$ For $H_{i}<H'_{i}\in(0,1)$, a kernel of the form $g_{i}(u)=(-u)_+^{H_{i}-\frac{1}{2}}+(-u)_+^{H'_{i}-\frac{1}{2}}$ satisfies \eqref{C2} {if $H_i>(H_i'-\tfrac{1}{2})\vee 0$ (this is to ensure that $\zeta<\tfrac{1}{2}-(\alpha)_-$, cf. the link between $H_i, H_i'$ and $\alpha,\zeta$ in the remark below)}. It yields a process $G^{(i)}$ with the local regularity of an $H_{i}$-fBm and the long-range dependence of an $H'_{i}$-fBm.
\end{example}
\begin{remark}\label{rk:C2}
We make a few comments on \eqref{C2}:\\
\eqref{C2supp} is equivalent to say that there exists $i$ such that $g_i$ is not zero on a set of positive Lebesgue measure. It is a fairly natural condition that ensures that the law of $G^{(i)}$ has full support in $\mathcal{C}(\R_+)$ (see \cite{Cherny}). \\
Besides, up to a shift in the definition of $\kerG$, we can assume that $\text{supp} (\kerG) \cap [-1,0] \neq \emptyset$. \\
\eqref{C2reg} refers to the memory of the noise process, and $\alpha$ plays an important role in our theorems. For instance, let us consider a one-dimensional fBm of parameter $H$. In this case, $g''(u)=\rho_H(-u)^{H-\frac{5}{2}}$ so that $\alpha=\tfrac{1}{2}-H>-\tfrac{1}{2}$. Keeping in mind that the memory of the fBm increases with $H$, one can thus interpret the parameter $\alpha$ as follows: the weight of the memory decreases when $\alpha$ increases.\\
\eqref{C2Holder} implies that $G$ is a.s. H\"older-continuous with H\"older exponent depending on $\zeta$ (see further).\\
Note that $\zeta$ can be negative. Having in mind that $\kerG$ is integrable near $0^-$ (due to \eqref{eq:kerG}), \eqref{C2Holder} mostly states that $\kerG$ is not too ``pathological'' near $0^-$. The case of an $H$-fBm corresponds to $\zeta = \tfrac{1}{2}-H$.\\
Note also that $\zeta$ does not appear in the rate of convergence $\gamma$ (see Theorem \ref{th:maingen2} below).
\end{remark}
By Proposition \ref{prop:contSDSSS}, under \eqref{C1} and \eqref{C2}, SDE \eqref{eds} admits almost surely a unique solution owing to the $a.s.$ continuity of $(G_t)_{t\ge0}$. The definition of invariant distribution is similar to the one recalled in the fractional case.  The idea is to build a \textit{stochastic dynamical system} over the SDE through the moving-average representation \eqref{eq:def_noise} of $(G_t)_{t\ge0}$ (which corresponds to the Mandelbrot- Van Ness representation  when $G$ is a fBm) and this way, to embed $(X_t)_{t\ge0}$ into a Feller Markov process on the product space $\ER^d\times{\cal W}$, where ${\cal W}$
  denotes an appropriate H\"older space. We go back to this construction in Appendix \ref{App:InvDistrib}. Then, for the existence of invariant distribution in the general case,
  we refer to Proposition \ref{prop:InvProbab} where we prove that existence holds under \eqref{C1} and \eqref{C2}. As concerns the uniqueness, it will be given by the main theorem (the coupling method used to evaluate the rate of convergence is also a way to prove uniqueness of invariant distribution).
We are now able to provide the extension of Theorem \ref{th:maingen} to the general case.
\begin{theorem}\label{th:maingen2}
Let $q\geq 1$. Let $X$ be a solution to (\ref{eds})  satisfying $\ES[|X_0|^{q+\upsilon}]<+\infty$ for a positive $\upsilon$. Assume \eqref{C1} and \eqref{C2}. Then, existence and uniqueness holds for the invariant distribution $\nu$. Furthermore,  the conclusions of Theorem \ref{th:maingen}
hold true with $$\gamma = \frac{2}{3}\left(\alpha+\frac{1}{2}-\varepsilon\right).$$
\end{theorem}
This result thus emphasizes that the rate of convergence depends mainly on the long-time parameter $\alpha$. When the process is a fBm, one retrieves Theorem \ref{th:maingen}
since in this case, $\alpha=\tfrac{1}{2}-H$.

\begin{remark}\label{rk:indepComp1}
Following carefully the proof of this theorem, one can see that it is still true for a noise process with dependent components. In particular, the formulation of Assumption \eqref{C2} is valid for this more general setting. The main nontrivial modification concerns the support of the law of $G$ in that case. This question is addressed in Remark \ref{rk:indepComp2}.
\end{remark}

Now, let us focus on the generalisation of Theorem \ref{th:maingenTV}, which reveals an additional difficulty. Actually, the proof of Theorem \ref{th:maingenTV}
is based on an explicit construction of the coupling of the fBms which in turns implies to be able to build the coupling between the underlying Brownian Motions (of the Mandelbrot-Van Ness representation). More precisely, this comes down to solve the following problem: for a given kernel $g:(-\infty,0)\mapsto\ER$ and a given (smooth enough) function $\varphi$, find a function $\Psi$ such that:
 \begin{equation}\label{eq:operatorliouville}
  \varphi(t)=\frac{d}{dt}\left(\int_0^t g(s-t) \Psi(s) ds\right), \quad t> 0 .
  \end{equation}
 When $g(t)=(-t)^{H-\frac{1}{2}}$, this equation has an explicit solution (see Lemma 4.2 of \cite{hairer}) given by 
 \begin{equation}\label{djslkjdql}
 \Psi(t)=c_H\frac{d}{dt}\left(\int_0^t (t-s)^{\frac{1}{2}-H} \varphi(s) ds\right), \quad t>0,
 \end{equation}
 where $c_H$ is a real constant.
  In other words, one is able to invert  explicitly the operator related to \eqref{eq:operatorliouville} in this case. 
In the general case, a way to overcome this absence of explicit form would be to prove the invertibility of the  operator and to provide some related properties, which is \emph{a priori} a difficult problem. In Subsection \ref{subsec:Laplace}, we give heuristics on how to find $\Psi$ in general. However we choose here to only provide a set of \textit{ad hoc} conditions which are sufficient to extend Theorem \ref{th:maingenTV}: 

\smallskip

\noindent \begin{minipage}{0.1\linewidth}
(\customlabel{C3}{$\C3$}):
\end{minipage}
\begin{minipage}[t]{0.9\linewidth}
\raggedright
For each $i\in\{1,\ldots,d\}$, for any ${\cal C}^1$-function $\varphi:(0,+\infty)\mapsto\ER$, there exists a function $\Psi_\varphi:(0,+\infty)\rightarrow \ER$ such that 
\begin{equation}\label{eq:linkPhiPsi}
 \varphi(t)=\frac{d}{dt}\left(\int_0^t g_i(s-t) \Psi_\varphi(s) ds\right), \quad t>0 ,
\end{equation}
and  there exists a ${\cal C}^1$-function $h_i:(-\infty,0)\mapsto \ER$ such that for any $\varphi\in\mathcal{C}^1((0,+\infty);\R)$, $\Psi_\varphi$ is given by 
\begin{equation}\label{eq:defpsi}
  \Psi_\varphi(t)=\frac{d}{dt}\left(\int_0^t h_i(s-t) \varphi(s) ds\right), \quad t>0.
\end{equation}
Moreover, one of the two following statements holds true:
 \begin{enumerate}[]
 \item (\customlabel{C3i}{$\C3_i$}) $\lim_{t\rightarrow0} h_i(t)=0$ and $h_i'$ is integrable on $[-1,0)$.
 \item (\customlabel{C3ii}{$\C3_{ii}$}) $h_i$ belongs to $L^2([-1,0])$ {and $b$ is Lipschitz continuous.}
 \end{enumerate}
\end{minipage}

\begin{remark} When $g_{i}(t)=(-t)^{H-\frac{1}{2}}$, the conjugate function $h_{i}$ is defined by $h_{i}(t)=c_H (-t)^{\frac{1}{2}-H}$ (by Equation \eqref{djslkjdql}). When $H<1/2$, Assumption \eqref{C3i} holds whereas, under the additional condition that $b$ is Lipschitz continuous, \eqref{C3ii} holds when $H>1/2$.
\end{remark}

We then have the following result:
\begin{theorem}\label{th:maingenTV2}
Let the assumptions of Theorem \ref{th:maingen2} be in force with $q=1$. Furthermore, assume \eqref{C3}. Then, \smallskip

\noindent (i) The conclusions of Theorem \ref{th:maingenTV}(i) hold true 
 with $\gamma =\frac{2}{3}\left(\alpha+\frac{1}{2}-\varepsilon\right).$\smallskip

\noindent (ii) If furthermore, for every $i\in\{1,\ldots,d\}$, $h'_i$ belongs to $L^2((-\infty,-1])$ then,   the conclusion of Theorem \ref{th:maingenTV}(ii) also holds true (with $\gamma=\frac{2}{3}\left(\alpha+\frac{1}{2}-\varepsilon\right)$).
\end{theorem}

\section{Overview of the proof of the theorems}\label{sec:overview}

\subsection{Decomposition of the driving process}\label{sec:decompo-fbm-ter}

To understand the memory structure of the Gaussian process $(G_t)_{t\ge0}$, one can consider the Mandelbrot-Van Ness representation equivalent to \eqref{eq:def_noise}, given by:
\begin{equation*}
G_t=\int_{-\infty}^{0} \left\{\kerG(u-t) - \kerG(u)\right\} ~dW_{u} + \int_0^t \kerG(u-t) ~dW_u,\quad t\ge0,
\end{equation*}
where $(W_t)_{t\in\R}$ is a two-sided $\R^d$-valued Brownian motion and $\kerG(u)$ is a diagonal matrix with entries $g_{i}(u)$ satisfying \eqref{C2}. 
This representation immediately gives rise to the decomposition 
\begin{align}\label{eq:decompo-fbm-ter}
\forall t,\tau \in \R_+,\quad G_{t+\tau}-G_\tau&= \int_{-\infty}^{\tau} \{ \kerG(u-(t+\tau)) - \kerG(u-\tau)\} \, dW_u  + \int_{\tau}^{t+\tau} \kerG(u-(t+\tau)) ~ dW_u \nonumber\\
&=: D_t(-\infty,\tau)+Z_t(\tau) \ ,
\end{align}
where the process $D$ is seen a the ``past'' component encoding the ``memory'' of $W$, while $Z$ stands for the ``innovation'' process (when looking at $G$ after time $\tau$). For given $\tau> \theta\geq -\infty$ and $\Delta\geq 0$, we subdivide $D$ into $(D_t(\theta,\tau))_{t\ge0}$ and $(D_t^\Delta(\theta))_{t\ge0}$ respectively defined for all $t\ge0$ by 
\begin{equation}\label{eq:defDt}
D_t(\theta,\tau)= \int_{\theta}^{\tau} \{ \kerG(u-(t+\tau)) - \kerG(u-\tau)\} ~ dW_u ,
\end{equation}
and 
\begin{equation*}
D_t^\Delta(\theta) = \int_{-\infty}^{\theta} \{ \kerG(u-(t+\theta+\Delta))-\kerG(u-(\theta+\Delta))\} ~ dW_u .
\end{equation*}
Hence for $\Delta = \tau-\theta$, $G$ reads
\begin{equation*}
G_{t+\tau}-G_{\tau}=D_t^\Delta(\theta) + D_t(\theta,\tau) +Z_t(\tau) ~.
\end{equation*}
With an adequate choice of $\theta$ and $\tau$, this is the decomposition of the noise between ``remote'' past and ``recent'' past that we shall use. Finally, the components of the previously defined processes are denoted by $D_t^{\Delta,(i)}(\theta)$, $D_t^{(i)}(\theta,\tau)$ and $Z_t^{(i)}(\tau)$, $i\in\{1,\dots,d\}$.

\subsection{Strategy of proof for the convergence  in Wasserstein distance}

This subsection gives an overview of the proof of Theorem \ref{th:maingen2} (from which Theorem \ref{th:maingen} is a consequence in the special case of fractional Brownian motion). We already pointed out that existence and uniqueness hold for the invariant measure $\nu$ of $(X_t,(G_{s+t})_{s\leq 0})_{t\geq 0}$, where $X$ is the solution to \eqref{eds} and $(G_t)_{t\in\ER}$ denotes a Gaussian process of the form (\ref{eq:def_noise}) satisfying \eqref{C2}.

Now consider a synchronous coupling of the SDE \eqref{eds}: Let $(X_t,Y_t)_{t\ge0}$ denote a solution to the following SDE in $\ER^d\times\ER^d$ 
\bqn\label{eq:eqcoup}
\begin{cases}
dX_t= b(X_t)dt+\sigma dG_t\\
dY_t= b(Y_t)dt+\sigma dG_t
\end{cases}
\eqn
with \emph{generalised} initial condition 
\begin{equation*}
\tilde{\mu}(dx_1,dx_2,dw)=\mu_1(w,dx_1)\mu_2(w,dx_2)\PE_G(dw)
\end{equation*}
where $\PE_G$ denotes the distribution of $(G_t)_{t\le0}$ on a H\"older-type space ${\cal H}_{\tilde{\rho};\zeta-\alpha}$
 (see Appendix \ref{App:InvDistrib}) and the transitions probabilities $\mu_1(\cdot,dx)$ and $\mu_2(\cdot,dy)$ correspond respectively to the conditional distributions of $X_0$ and $Y_0$ given $(G_t)_{t\le0}$.
Furthermore, one assumes that the law of $X_0$, here denoted by $\mu_1$, satisfies the moment condition given in Theorems \ref{th:maingen} to \ref{th:maingenTV2}, and  $\mu_2\otimes \PE_G=\nu$ where $\nu$ denotes the unique invariant distribution. Hence, $(Y_t)_{t\ge0}$ is stationary and in particular, ${\cal L}(Y_t)={\bar{\nu}}$ where ${\bar{\nu}}$ denotes the first marginal of the invariant distribution $\nu$. As a consequence, for $q\geq 1$,
\begin{align}\label{eq:boundWass}
W_{q}({\cal L}(X_t),\bar{\nu})\le \ES[|X_t-Y_t|^q]^{\frac{1}{q}}
\end{align}
and the strategy is now to control the right-hand side of the previous inequality.

To this end, let $(\tau_k)_{k\in\N}$ be any non-decreasing sequence of stopping times. The following inequality is the starting point of our proof. Assuming that the expectations below are finite, we have for all $t\geq 1$, 
\begin{align}\label{eq:main_eq}
\EE\left[|X_t - Y_t|^{q}\right] &= \sum_{k\in\N} \EE\left[|X_t - Y_t|^q ~\mathbf{1}_{[1+\tau_k,1+\tau_{k+1})}(t) \right]  \nonumber\\
&\leq \sum_{k\in\N}  \EE\left[|X_{1+\tau_k} - Y_{1+\tau_k}|^q ~\mathbf{1}_{[1+\tau_k,1+\tau_{k+1})}(t)\right]\nonumber\\
&\leq \sum_{k\in\N}  \EE\left[|X_{1+\tau_k} - Y_{1+\tau_k}|^{2q}\right]^{\frac{1}{2}} \times \PP\left(t\in [1+\tau_k,1+\tau_{k+1})\right)^{\frac{1}{2}}\nonumber\\
&\leq \sum_{k\in\N}  \EE\left[|X_{1+\tau_k} - Y_{1+\tau_k}|^{p}\right]^{\frac{1}{2}} \times \PP\left(\tau_{k+1}>t-1\right)^{\frac{1}{2}},
\end{align}
denoting $p=2q \in [2,+\infty)$.\\
In Section \ref{sec:waitingtimes}, we build an increasing sequence of stopping times $(\tau_k)_{k\ge1}$ such that $\forall k\ge1$,
\begin{equation}\label{memorycondition}
\begin{split}
&\Delta_{k+1} := \tau_{k+1}-\tau_{k}-1 \geq 1 \\
&\| D^{\Delta_{k+1}} (1+\tau_{k})\|_{\infty,[0,1]} \leq K_{R} ~~\text{a.s.}\\
&\PP\left(\|D(1+\tau_{k},\tau_{k+1})\|_{\infty,[0,1]}\le K_r |{\cal F}_{1+\tau_k}\right) \geq \tfrac{1}{2} ~~\text{a.s.},
\end{split}
\end{equation}
where $K_R, K_r>0$ are independent of $k$ ($K_R$ refers to the Remote past, while $K_r$ is for the recent past).

\noindent Condition \eqref{memorycondition} means that at time $\tau_{k+1}$, the supremum norm of the memory term (see the decomposition introduced in Subsection \ref{sec:decompo-fbm-ter}) is bounded with positive probability (conditionally to ${\cal F}_{1+\tau_k}$). In particular, notice that the remote past is controlled deterministically, which will be crucial in Section \ref{sec:contraction}. Roughly, the consequence is that the dynamics of the SDE between $\tau_k$ and $\tau_{k+1}$ is not so far from a standard diffusion perturbed by a controlled drift term. Such a property is certainly of interest if one is able to obtain some probabilistic bounds on the sequence $(\tau_k)_{k\ge1}$. More precisely, one can build a sequence $(\tau_k)_{k\ge1}$ such that the condition \eqref{memorycondition} holds and such that for $\lambda>0$ and $r>0$, there exists $C_{\lambda,r}>0$ such that for all $k\in\N$,
\begin{align}\label{eq:boundMomentsTau}
\EE\left[\exp\{\lambda \tau_k^r\} \right] \leq C_{\lambda,r}^k~,
\end{align}
with the property that $\lim_{\lambda \rightarrow 0} C_{\lambda,r}=1$. This is the aim of Section \ref{sec:waitingtimes}.\smallskip

With such a rough view, one hopes to obtain a contraction property between $\tau_{k-1}$ and $\tau_{k}$. More precisely, we shall prove that
\begin{align}\label{contraction}
\ES[|X_{1+\tau_{k}}-Y_{1+\tau_{k}}|^p]& \le \varrho\ES[|X_{\tau_{k}}-Y_{\tau_{k}}|^p]\nonumber\\
&\le \varrho\ES[|X_{1+\tau_{k-1}}-Y_{1+\tau_{k-1}}|^p]
\end{align}
where $\varrho$ lies in $(0,1)$ (and is independent of $k$ and $p$). Establishing such a property will be the purpose of Section \ref{sec:contraction} below. The fundamental idea there is to send $X$ far enough from the origin, in a region where exponential contraction happens independently of the position of $Y$. This is achieved using the support of the process $(Z_t(\tau))_{t\in[0,1]}$ defined in \eqref{eq:decompo-fbm-ter}, so that reaching this region happens with positive probability. \smallskip

The final step in the proof of Theorem \ref{th:maingen2} happens in Section \ref{sec:proofTh3}. In view of \eqref{eq:boundMomentsTau}, Markov's inequality applied to $\PP\left(\tau_{k+1}>t-1\right)$ yields a sub-exponential rate of decay $r$ in time. Combined with \eqref{contraction} and injected in \eqref{eq:main_eq} and then \eqref{eq:boundWass}, yields the expected result. The choice of $\lambda$ is optimized in order to get $r$ as large as possible.

\subsection{Strategy of proof for the convergence in total variation distance}

From the definition of the total variation distance, we have the following inequality:
\begin{align*}
\forall t\in\R_+,\quad \|\nu_{t} - \bar{\nu}\|_{TV} \leq 2\PP\left(X_{t} \neq Y_t\right). 
\end{align*}
Using the synchronous coupling of the noises used so far up to time $t-1$, we have seen that we are able to control the $L^2$-distance  between $X$ and $Y$  and hence,
to lower-bound the probability that $X$ and $Y$ be close at time $t-1$.
This coupling is very convenient as it is in some sense ``free of the past''.  Then, when $X_{t-1}$ and $Y_{t-1}$ are close, the idea to get bounds in total variation is to show that the cost of the coalescent coupling between $t-1$ and $t$ is ``small'' (or equivalently, the probability that $X_t=Y_t$ is high). This part is achieved using a Girsanov-type argument close to \cite{hairer}: one exhibits a (random) function $\varphi$ defined on $[t-1,t]$ such that if the driving 
Gaussian processes $G$ and $\widetilde{G}$ of $X$ and $Y$ satisfy (on a subset $\Omega_1$ of $\Omega$)
\begin{align*}
\widetilde{G}_s = 
\begin{cases}
G_s & \text{ if } s\leq t-1\\
G_s + \int_0^s \varphi(u)~du & \text{ if } s\in (t-1,t],
\end{cases}
\end{align*}
then the paths stick at time $t$. Then, the Girsanov theorem is applied on the underlying Wiener processes involved by $G$ and $\widetilde{G}$ (this step requires Assumption \eqref{C3} in the general case)
and an optimization of the parameters shows that the order of the Wasserstein rate of convergence is preserved in total variation.\smallskip

\noindent To  extend the result to the functional setting (\eqref{eq:functionalTV} in the fractional case), the additional step is to show that the (non-trivial) coupling which is necessary to
preserve that $X$ and $Y$ stay together after time $t$ is also small when  when $X_{t-1}$ and $Y_{t-1}$ are close.

\section{Construction and properties of $(\tau_k)_{k\in\N}$}\label{sec:waitingtimes}

The aim of this section is to exhibit a sequence of stopping times which satisfies (\ref{memorycondition}). We also obtain that the probability tails of these stopping times decrease with a sub-exponential rate.

\subsection{Properties of $\kerG$}

Recall that the kernel $\kerG$ is always assumed to satisfy the $L^2$ condition \eqref{eq:kerG}, otherwise the noise representation \eqref{eq:def_noise} cannot make sense. First, we give the following simple consequences of \eqref{C2}.

\begin{lemma}\label{lem:gprime}
Let $\kerG$ be an $\Md$-valued function satisfying \eqref{C2} and $W$ be the a.s. continuous version of a two-sided $\R^d$-valued Brownian motion. Then 
\begin{align*}
\lim_{r\rightarrow -\infty} \kerG'(r) = 0 .
\end{align*}

\end{lemma}
\begin{proof}
Note that the proof reduces to a one-dimensional problem, since it suffices to prove the above claims for all the diagonal elements of $\kerG$ independently. Hence, let $g$ be any of the diagonal entries of $\kerG$ and remark that $g$ satisfies \eqref{C2}.

Observe first that for any sequence $(r_n)_{n\in\N} \subset (-\infty,-1]$ that diverges to $-\infty$, $(g'(r_n))_{n\in\N}$ is a Cauchy sequence. Indeed, \eqref{C2reg} provides the following bound:
\begin{align*}
|g'(r_n) - g'(r_m)| = |\int_{r_m}^{r_n} g''(u) \dd u| &\leq C_1 \int_{r_m}^{r_n} (-u)^{-\alpha-2} \dd u \\
&\leq \frac{C_1}{1+\alpha} |r_n^{-\alpha-1} - r_m^{-\alpha-1}|,
\end{align*}
where we recall that $-\alpha-1<-\tfrac{1}{2}$. Thus denote by $g'_\infty$ the limit of $g'$ at $-\infty$, and let us prove that $g'_\infty =0$. With the result of the above paragraph and the integrability of $g''$ at $-\infty$, one gets
\begin{align*}
\forall r\leq -1,\quad g'(r) = g'_\infty + \int_{-\infty}^r g''(u) \dd u. 
\end{align*}
Integrating once more,
\begin{align*}
\forall r,s\leq -1,\quad |g(r) - g(s)| &= |(r-s)g'_\infty + \int_s^r \int_{-\infty}^v g''(u) \dd u \dd v | \geq |(r-s)g'_\infty| - |\int_s^r \int_{-\infty}^v g''(u) \dd u \dd v | .
\end{align*}
For some fixed $t>0$ and $s=r-t$, observe that 
\begin{align*}
|\int_s^r \int_{-\infty}^v g''(u) \dd u \dd v | &\leq \frac{C_1}{|\alpha|(1+\alpha)} \left|(-r)^{-\alpha}-(t-r)^{-\alpha}\right|\\
&\leq \frac{C_1}{|\alpha|(1+\alpha)} (-r)^{-\alpha} \left|1-(1-\tfrac{t}{r})^{-\alpha}\right| .
\end{align*}
For further use, we note that for any $\alpha>-\tfrac{1}{2}$ ($\alpha\neq 0$) and $t>0$,
\begin{align}\label{incGL2}
(-r)^{-\alpha} \left|1-(1-\tfrac{t}{r})^{-\alpha}\right| \underset{r\rightarrow -\infty}{\sim} \alpha t (-r)^{-(1+\alpha)}
\end{align}
and that $r\mapsto (-r)^{-\alpha} \left(1-(1-\tfrac{t}{r})^{-\alpha}\right) \in L^2\left((-\infty, -1]\right)$. In particular, if $g'_\infty\neq 0$, then for any $t>0$,
\begin{align*}
|g(r) - g(r-t)| \underset{r\rightarrow -\infty}{\sim} t |g'_\infty|
\end{align*}
which is not compatible with \eqref{eq:kerG}. Hence, $g'_\infty= 0$.
\end{proof}

\begin{lemma}\label{lem:conseq_C2}
Let $\kerG$ be an $\Md$-valued function satisfying \eqref{C2} and $W$ be the a.s. continuous version of a two-sided $\R^d$-valued Brownian motion. Then
\begin{enumerate}[a)]
\item $\forall T\geq 0,~\forall t\in [0,1],\quad \displaystyle\lim_{r\rightarrow-\infty} \left\{\kerG(r-(t+T))-\kerG(r-T)\right\} W_r = 0$ a.s. ;
\item $\forall T\geq 1,~\forall t\in [0,1]$, $\left(\left\{\kerG'(r-(t+T))-\kerG'(r-T)\right\} W_r\right)_{r\leq 0}$ is integrable on $\R_-$ a.s. ;
\item $\forall r\leq -1,\quad \|\kerG'(r)\| \leq \frac{C_1}{\alpha+1} (-r)^{-(\alpha+1)} $ and there exists $C>0$ such that $\forall r\leq -1$:
\begin{align*}
\|\kerG(r)\| \leq C \left(1+(-r)^{-\alpha} \right).
\end{align*}
\item $\exists C>0$ such that $\forall r\in[-2,0)$,
\begin{align*}
\|\kerG'(r)\| \leq C \left( 1+(-r)^{-\zeta-1} \right) \quad \text{and} \quad \|\kerG(r)\|\leq C \left(1+ (-r)+(-r)^{-\zeta}  \right).
\end{align*}
\end{enumerate}
\end{lemma}

\begin{proof}
As in the previous proof, it suffices to prove the above claims for all the diagonal elements of $\kerG$ independently. Hence, let $g$ be any of the diagonal entries of $\kerG$ and remark that $g$ satisfies \eqref{C2}. \\
Starting with the proof of $a)$, we have from Lemma \ref{lem:gprime} that $g'(r) = \int_{-\infty}^r g''(u) \dd u$ for $r\leq -1$, so 
\begin{align*}
|g(r-(t+T))-g(r-T)| = |\int_{r-T}^{r-(t+T)} \int_{-\infty}^v g''(u) \dd u \dd v|
\end{align*}
and in view of \eqref{incGL2}, this quantity is of order $(-r)^{-(1+\alpha)}$ in the neighbourhood of $-\infty$. Since $\alpha>-\tfrac{1}{2}$, this proves $a)$.

To prove $b)$, it follows from \eqref{C2reg} that
\begin{align*}
|g'(r-(t+T))-g'(r-T)| &\leq \sup_{u\in[r-(1+T),r-T]}|g''(u)|\\
&\leq C_1 (T-r)^{-\alpha-2} .
\end{align*}
Since $\alpha>-\tfrac{1}{2}$, it is clear that $\left(\left(g'(r-(t+T))-g'(r-T)\right) W_r\right)_{r\leq 0}$ is integrable.

The proof of the first part of $c)$ follows again from Lemma \ref{lem:gprime} and \eqref{C2reg}. For the second point, use again Lemma \ref{lem:gprime} to get that
\begin{align*}
\forall r\leq -1,\quad |g(r)| \leq |g(-1)| + |\int_r^{-1} \int_{-\infty}^v g''(u) \dd u \dd v | \leq C\left(1 + (-r)^{-\alpha} \right) .
\end{align*}

The inequalities of $(d)$ are consequences of \eqref{C2Holder}:
\begin{align*}
|g'(r)| = |g'(-2) - \int_r^{-2} g''(u)~du| &\leq C(1+(-r)^{-\zeta-1}).
\end{align*}
The bounds on $g$ follow by exactly the same method.
\end{proof}

\subsection{Construction}

We propose an iterative construction of the stopping times. First, fix $\tau_0 =0$ (and use the convention $\tau_{-1}=-\infty$) and assume that for $k\geq 1$, $\tau_1,\dots\tau_{k-1}$ have been constructed.\\
With the constant $\alpha>-\tfrac{1}{2}$ from \eqref{C2}, let us set, for $\epsilon\in(0,\alpha+\tfrac{1}{2})$ and $k\in \N^*$:
\begin{align*}
S^{k-1,\epsilon} = \sup_{s\in(1+\tau_{k-2},1+\tau_{k-1}]} \frac{|W_s-W_{1+\tau_{k-1}}|}{(2+\tau_{k-1}-s)^{\frac{1}{2}+\epsilon}} .
\end{align*}
Then, for $\delta>0$ and $\chi>0$ that will be calibrated later, let us define
\begin{align}\label{eq:Delta}
\Delta_k = k^\chi + \left(S^{k-1,\epsilon} \right)^{\frac{1}{\delta}} .
\end{align}
Finally, set $$\tau_k = 1 + \tau_{k-1} + \Delta_k .$$ 
Let $\mathcal{F}_t  = \sigma\left(W_s,s\in(-\infty,t]\right), t\geq 0$, denote the natural filtration of the two-sided Brownian motion. Observe that $\Delta_k$ is $\mathcal{F}_{1+\tau_{k-1}}$-measurable so that $\tau_k$ is $\mathcal{F}_{1+\tau_{k-1}}$-measurable. Using that for some deterministic $t_1$ and $t_2$ with $0\le t_1\le t_2$, $(W_s-W_{t_2})_{s\in[t_1,t_2]}$ has the same distribution than $(W_s)_{s\in[0,t_1-t_2]}$, this implies that conditionally to ${\cal F}_{1+\tau_{k-2}}$, 
$${\cal L}(S^{k-1,\epsilon}|{\cal F}_{1+\tau_{k-2}})\overset{(d)}{=}\sup_{s\in [0,{1+}\Delta_{k-1}]} \frac{|\widetilde{W}_s|}{(1+ s)^{\frac{1}{2}+\epsilon}} \leq \|\widetilde{W}\|_{\frac{1}{2}+\epsilon,\infty},$$
for some Brownian motion $\widetilde{W}$ independent of the sequence $(\Delta_k)_{k\in\N}$ (and $\|\widetilde{W}\|_{\frac{1}{2}+\epsilon,\infty}$ is defined in this norm in  Section \ref{subsec:notations}).

Of course the first condition of (\ref{memorycondition}) is satisfied for this construction of $(\tau_k)_{k\in\N}$. The next proposition shows that with this choice of $(\tau_k)_{k\in\N}$, the second condition of (\ref{memorycondition}) is also satisfied.
\begin{proposition}\label{prop:remotePast}
With the notations of Subsection \ref{sec:decompo-fbm-ter} and $(\tau_k)_{k\in\N}$ as above, assume that $\epsilon, \delta$ and $\chi$ are such that
\begin{equation}\label{cond:epsdeltachi}
\begin{cases}
&\alpha_{\epsilon,\delta} := \alpha+\frac{1}{2}-\epsilon-\delta \in (0,1) \\
&\chi {\ge} \alpha_{\epsilon,\delta}^{-1}-1 .
\end{cases}
\end{equation}
Then the following inequality holds for any $k\geq 1$, almost surely:
\begin{equation*}
\|D^{\Delta_k}(1+\tau_{k-1})\|_{\infty,[0,1]} \leq C_{\epsilon,\delta}~,
\end{equation*}
where $C_{\epsilon,\delta} =  \frac{C_1}{\alpha_{\epsilon,\delta}} + \bigg( \frac{C_1}{1+\alpha} \displaystyle\max_{k\in\N^*} \sum_{j=1}^{k-1} \bigg(\sum_{l=j}^k l^\chi\bigg)^{-\alpha_{\epsilon,\delta}} \bigg)$ is finite, and $C_1$ is the constant in \eqref{C2reg}.
\end{proposition}

\begin{proof}
Let $k\in\N^*$. We decompose $D^{\Delta_k}(1+\tau_{k-1})$ into the following sum:
\begin{align*}
\forall t\in[0,1],~D_t^{\Delta_k}(1+\tau_{k-1}) &= \int_{-\infty}^1 \left\{\kerG(u-(\tau_k+t)) - \kerG(u-\tau_k)\right\}~dW_u \\
&\quad + \sum_{j=1}^{k-1} \int_{1+\tau_{j-1}}^{1+\tau_{j}} \left\{\kerG(u-(\tau_k+t)) - \kerG(u-\tau_k)\right\}~dW_u  
\end{align*}
In view of the fact that $\displaystyle \lim_{u\rightarrow -\infty} \left\{\kerG(u-(\tau_k+t)) - \kerG(u-\tau_k)\right\}(W_u-W_{1+\tau_{j}}) = 0$ a.s. (see Lemma \ref{lem:conseq_C2} $a)$), we can integrate-by-parts in the previous equality: starting with the first term in the above equation, this reads
\begin{align*}
\forall t\in[0,1],~ \int_{-\infty}^1 \big\{\kerG(u-(\tau_k+t)) - \kerG(u-\tau_k)&\big\}~dW_u = \left[ \left\{\kerG(u-(\tau_k+t)) - \kerG(u-\tau_k)\right\} (W_u-W_{1})\right]_{-\infty}^{1}\\
&\quad -\int_{-\infty}^{1} \left\{\kerG'(u-(\tau_k + t)) - \kerG'(u-\tau_k)\right\} (W_u-W_{1}) ~du \\
&= -\int_{-\infty}^{1} \left\{\kerG'(u-(\tau_k + t)) - \kerG'(u-\tau_k)\right\} (W_u-W_{1})  ~du,
\end{align*}
and similarly for $j=1,\dots,k-1$,
\begin{align*}
\int_{1+\tau_{j-1}}^{1+\tau_{j}} \big\{\kerG(u-(\tau_k+t)) -& \kerG(u-\tau_k)\big\}~dW_u \\
&= \left[ \left\{\kerG(u-(\tau_k+t)) - \kerG(u-\tau_k)\right\} (W_u-W_{1+\tau_{j}})\right]_{1+\tau_{j-1}}^{1+\tau_{j}}\\
&\quad -\int_{1+\tau_{j-1}}^{1+\tau_{j}} \left\{\kerG'(u-(\tau_k + t)) - \kerG'(u-\tau_k)\right\} (W_u-W_{1+\tau_{j}}) ~du \\
&= \left\{\kerG(1+\tau_{j-1}-(\tau_k+t)) - \kerG(1+\tau_{j-1}-\tau_k)\right\} (W_{1+\tau_{j-1}}-W_{1+\tau_{j}}) \\
&\quad -\int_{1+\tau_{j-1}}^{1+\tau_{j}} \left\{\kerG'(u-(\tau_k + t)) - \kerG'(u-\tau_k)\right\} (W_u-W_{1+\tau_{j}})  ~du .
\end{align*}
With the convention $\tau_{-1} = -\infty$, we gather from the above expressions that $\forall t\in[0,1]$,
\begin{align}\label{eq:decompDDelta}
\begin{split}
D_t^{\Delta_k}(1+\tau_{k-1}) &= \sum_{j=1}^{k-1} \left\{\kerG(1+\tau_{j-1}-(\tau_k+t)) - \kerG(1+\tau_{j-1}-\tau_k)\right\} (W_{1+\tau_{j-1}}-W_{1+\tau_{j}}) \\
&\quad -\sum_{j=0}^{k-1}\int_{1+\tau_{j-1}}^{1+\tau_{j}} \left\{\kerG'(u-(\tau_k + t)) - \kerG'(u-\tau_k)\right\} (W_u-W_{1+\tau_{j}}) ~du .
\end{split}
\end{align}
Using Assumption \eqref{C2reg}, one gets that for any $u\in (1+\tau_{j-1},1+\tau_j)$,
\begin{align*}
|\kerG'(u-(\tau_k + t)) - \kerG'(u-\tau_k)| \leq t \sup_{r\in (0,1)} |\kerG''(u-\tau_k-r)| &\leq C_1 (\tau_k-u)^{-\alpha-2}.
\end{align*}
Since $j\leq k-1$ and by the simple inequalities $\tau_k\geq 2+\tau_{j-1}$ and $\tau_k-u\geq \Delta_{j+1}$, it follows that
\begin{align}\label{eq:boundG'}
|\kerG'(u-(\tau_k + t)) - \kerG'(u-\tau_k)|\leq C_1 (\tau_k-u)^{-\alpha-\frac{3}{2}+\epsilon+\delta}~ \Delta_{j+1}^{-\delta}~ (2+\tau_{j}-u)^{-\frac{1}{2}-\epsilon} .
\end{align}
Now by Lemma \ref{lem:conseq_C2} c), 
\begin{align*}
\left|\kerG(1+\tau_{j-1}-(\tau_k+t)) - \kerG(1+\tau_{j-1}-\tau_k)\right| &\leq t \sup_{r\in (0,1)} |\kerG'(1+\tau_{j-1}-(\tau_k+r))| \leq \frac{C_1}{1+\alpha} (\tau_k-\tau_{j-1}-1)^{-\alpha-1}.
\end{align*}
Hence the definition of the sequence $(\Delta_k)_{k\in\N}$ yields the inequalities $\tau_k-\tau_{j-1}-1 \geq 2+\Delta_{j}$, $\tau_k-\tau_{j-1}-1 \geq \Delta_{j+1}$ and $\tau_k-\tau_{j-1}-1 \geq \sum_{l=j}^k l^\chi$, so that
\begin{align}\label{eq:boundG}
\left|\kerG(1+\tau_{j-1}-(\tau_k+t)) - \kerG(1+\tau_{j-1}-\tau_k)\right|\leq \frac{C_1}{1+\alpha} (2+\Delta_{j})^{-\frac{1}{2}-\epsilon} ~ \Delta_{j+1}^{-\delta} ~\bigg(\sum_{l=j}^k l^\chi\bigg)^{-\alpha-\frac{1}{2}+\epsilon+\delta}.
\end{align}
Thus, plugging \eqref{eq:boundG'} and \eqref{eq:boundG} into \eqref{eq:decompDDelta}, one gets that
\begin{align*}
\|D^{\Delta_k}(1+\tau_{k-1})\|_{\infty,[0,1]} &\leq \frac{C_1}{1+\alpha} \sum_{j=1}^{k-1} \frac{S^{j,\epsilon}}{\Delta_{j+1}^\delta} ~\bigg(\sum_{l=j}^k l^\chi\bigg)^{-\alpha_{\epsilon,\delta}} + C_1 \sum_{j=0}^{k-1} \frac{S^{j,\epsilon}}{\Delta_{j+1}^\delta}  \int_{1+\tau_{j-1}}^{1+\tau_j} (\tau_k-u)^{-(\alpha_{\epsilon,\delta}+1)} ~du
\end{align*}
It is clear that $\frac{S^{j,\epsilon}}{\Delta_{j+1}^\delta}\leq 1$ by definition of $\Delta_{j+1}$, thus
\begin{align*}
\|D^{\Delta_k}(1+\tau_{k-1})\|_{\infty,[0,1]} &\leq \frac{C_1}{1+\alpha} \sum_{j=1}^{k-1} \bigg(\sum_{l=j}^k l^\chi\bigg)^{-\alpha_{\epsilon,\delta}} + C_1 \int_{-\infty}^{1+\tau_{k-1}} (\tau_k-u)^{-(\alpha_{\epsilon,\delta}+1)} ~du \\
&= \frac{C_1}{1+\alpha} \sum_{j=1}^{k-1} \bigg(\sum_{l=j}^k l^\chi\bigg)^{-\alpha_{\epsilon,\delta}} + \frac{C_1}{\alpha_{\epsilon,\delta}} \Delta_k^{-\alpha_{\epsilon,\delta}}.
\end{align*}
Since $\Delta_k>1$, it remains to prove that the first member is bounded in $k$ under the conditions \eqref{cond:epsdeltachi} on $\chi$ and  $\alpha_{\epsilon,\delta}$. First,
$$\bigg(\sum_{l=j}^k l^\chi\bigg) \geq \int_{j-1}^{k} t^\chi dt =
\frac{1}{\chi+1}\left( k^{\chi+1}- (j-1)^{\chi+1}\right)$$
and hence, for $k\geq 2$,
$$\sum_{j=1}^{k-1} \bigg(\sum_{l=j}^k l^\chi\bigg)^{-\alpha_{\epsilon,\delta}} \le \frac{k^{-\alpha_{\epsilon,\delta}(\chi+1)}}{(\chi+1)^{-\alpha_{\epsilon,\delta}}}
\sum_{j=1}^{k-1}\left(1-\left(\frac{j-1}{k}\right)^{\chi+1}\right)^{-\alpha_{\epsilon,\delta}}.$$
It follows that 
\begin{align*}
\limsup_{k\rightarrow +\infty} k^{\alpha_{\epsilon,\delta}(\chi+1)-1}\sum_{j=1}^{k-1} \bigg(\sum_{l=j}^k l^\chi\bigg)^{-\alpha_{\epsilon,\delta}}\leq \frac{1}{(\chi+1)^{-\alpha_{\epsilon,\delta}}}\int_0^1(1-t^{\chi+1})^{-\alpha_{\epsilon,\delta}} dt<+\infty
\end{align*}
since $\alpha_{\epsilon,\delta}<1$. Now with $\alpha_{\epsilon,\delta}(\chi+1)-1\ge0$, the expected result follows.
\end{proof}

The following technical lemma will be useful in the proof of the second main proposition of this section. Consider the process $(R_t)_{t\in[0,1]}$ defined by
\begin{equation}\label{eq:def_eq_R}
R_t:= \int_{-1}^0 \left(\kerG(u-t) - \kerG(u)\right) ~dW_u .
\end{equation}

\begin{lemma}\label{lem:boundR}
Under Assumption \eqref{C2}, the processes $G$ and $R$ have H\"older-continuous modifications on any interval $[T,T+1]$, $T>0$ (we shall assume from now on that $R$ and $G$ are these modifications): For any $\mathfrak{h} < \tfrac{1}{2}\wedge (\tfrac{1}{2}-\zeta)$, there exists a random variable $M$ with moments of all order such that
\begin{align*}
\PP\left(\forall s, t\in[T,T+1],~|G_s-G_t|\leq M|t-s|^{\mathfrak{h}}\right) = 1.
\end{align*}
The same conclusion holds for $R$. In particular, for any $\eta\in(0,1)$, there exists $K>0$ such that
\begin{equation*}
\PP\left(\|R\|_{\infty,[0,1]} \leq K\right)\geq 1-\eta .
\end{equation*}
\end{lemma}

\begin{proof}
First, let us remark that it is enough to prove that for any $t\in[0,1]$,
\begin{align}\label{eq:varG}
\EE\left[|G_t|^2\right] \leq C\left(t^{1-2\zeta} + t\right) .
\end{align}
Actually, if \eqref{eq:varG} holds, the increment stationarity and the Gaussian property of $G$ imply that for any $p\geq 1$ and any $s,t\in[T,T+1]$,
\begin{align*}
\EE\left[|G_t-G_s|^{2p}\right] \leq C_p\left(|t-s|^{1-2\zeta} + |t-s|\right)^p .
\end{align*}
Hence by Kolmogorov's continuity criterion, $G$ has a H\"older-continuous modification of any order $\mathfrak{h} < \tfrac{1}{2}\wedge (\tfrac{1}{2}-\zeta)$ and $M$ (which depends on $\mathfrak{h}$) satisfies $\EE[M^k]<\infty,~\forall k\in\N$.

It is then clear that the bound \eqref{eq:varG} also holds for $R$, and so the H\"older continuity as well. Besides, since the random variable $M$ has a finite first moment, it follows that
\begin{equation*}
\EE\left[\|R\|_{\infty,[0,1]}\right]<\infty.
\end{equation*}
Thus the desired inequality follows from Markov's inequality with $K>\frac{\EE\left[\|R\|_{\infty,[0,1]}\right]}{\eta}$.

\smallskip

We now prove Inequality \eqref{eq:varG}. It is enough to prove it for a single component of $G$, so fix $i\in\{1,\dots,d\}$:
\begin{align*}
\EE\left[|G_t^{(i)}|^2\right] &= \int_{-\infty}^t \left\{g_i(u-t) - g_i(u)\right\}^2~du \\
&= \int_{-\infty}^{-t} \left\{g_i(u-t) - g_i(u)\right\}^2~du+\int_{-t}^0 \left\{g_i(u-t) - g_i(u)\right\}^2~du+\int_{0}^t g_i(u-t)^2~du \\
&=: I_1+I_2+I_3.
\end{align*}
For $I_1$, Lemma \ref{lem:conseq_C2} c) and d) are used to get:
\begin{align*}
I_1 &\leq C t^2 \int_{-\infty}^{-1}  (-u)^{-2(\alpha+1)}~du + C t^2 \int_{-1}^{-t}  \left(C+(-u)^{-\zeta-1}\right)^2 ~du \\
&\leq C \left(t^2 + t^{1-2\zeta}\right).
\end{align*}
Now for $I_3$, Lemma \ref{lem:conseq_C2} d) implies that:
\begin{align*}
I_3 &\leq C \int_0^t (1+u^2 + u^{-2\zeta})~du = C(t+ \frac{1}{3}t^3+ \frac{1}{1-2\zeta} t^{1-2\zeta}).
\end{align*}
Since we assumed that $t\in[0,1]$, this always yields $I_3\leq C (t+ t^{1-2\zeta})$. Finally, $I_2$ is bounded similarly to $I_3$ since:
\begin{align*}
I_2 &\leq 2 \int_{-2t}^{-t} g_i(u)^2~du + 2\int_{-t}^{0} g_i(u)^2~du
\end{align*}
and this gives the expected result.
\end{proof}

\begin{remark}
For negative values of $\zeta$, it is possible that $R$ is more than $\tfrac{1}{2}^-$-H\"older continuous. In fact, for $\zeta$ small enough, one can deduce that $g'(0)$ and $g(0)$ take finite values. If this value is $0$, then indeed $R$ will be $(\tfrac{1}{2}-\zeta)^-$-H\"older continuous. However, it is enough for our purpose to get $\tfrac{1}{2}^-$-H\"older continuity, which is why we do not make this distinction.
\end{remark}

We are now ready to prove that the third condition of (\ref{memorycondition}) is satisfied. 
\begin{proposition}\label{prop:recentPast}
With the notations of Subsection \ref{sec:decompo-fbm-ter}, for any $\eta\in(0,1)$, there exists $K\in\R_+$ such that the following inequality holds true for any $(T_0,T_1)\in\ER_+^2$ with $ T_1-T_0\ge 1$:
\begin{equation*}
\PP\left(\|D(T_0,T_1)\|_{\infty,[0,1]} \leq K \right) \geq 1-\eta~.
\end{equation*}
\end{proposition}

\begin{proof}
We divide $D(T_0,T_1)$ into two parts:
\begin{align}\label{eq:recent_past_decomp}
D_t(T_0,T_1) &= \int_{T_0}^{T_1-1} \left\{\kerG(u-(T_1+t)) - \kerG(u-T_1)\right\} ~dW_u + \int_{T_1-1}^{T_1} \left\{\kerG(u-(T_1+t)) - \kerG(u-T_1)\right\} ~dW_u  \nonumber\\ 
&=: D^1_t(T_0,T_1-1) + D_t(T_1-1,T_1) .
\end{align}
These two components are independent and hence, for any positive $K_1$ and $K_2$ with $K_1+K_2\le K$,
$$\PP\left(\|D(T_0,T_1)\|_{\infty,[0,1]} \leq K \right)\ge \PP\left(\|D^1(T_0,T_1-1)\|_{\infty,[0,1]} \leq K_1 \right)
\PP\left(\|D(T_1-1,T_1)\|_{\infty,[0,1]} \leq K_2 \right).$$
It is thus enough to show that some $K_1$ and $K_2$ exist such that the two right-hand side terms are greater than $\sqrt{1-\eta}$, independently of $T_0$ and $T_1$. We prove it separately.

~

\textbf{$1^{\text{st}}$ step.} Set $\Delta=T_1-T_0$. By integration-by-parts, the first term in the RHS of the previous equality reads
\begin{align*}
D_t^1(T_0,T_1-1) &= -\left\{\kerG(-\Delta-t) - \kerG(-\Delta)\right\} (W_{T_0}-W_{T_1-1})  \\
&\quad- \int_{T_0}^{T_1-1} \left\{\kerG'(u-(T_1+t))-\kerG'(u-T_1)\right\} (W_u-W_{T_1-1}) ~du .
\end{align*}
We deduce the following upper bound:
\begin{align*}
|D_t^1(T_0&,T_1-1)| \leq t \sup_{r\in [-\Delta-t,-\Delta]}|\kerG'(r)| |W_{T_0}-W_{T_1-1}| \nonumber\\
&\quad+ t \sup_{r\in[T_0,T_1-1]} \frac{|W_r-W_{T_1-1}|}{(T_1-r)^{\frac{1}{2}+\epsilon}} \int_{T_0}^{T_1-1} (T_1-u)^{\frac{1}{2}+\epsilon} \sup_{v\in[u-(T_1+t),u-T_1]}|\kerG''(v)| ~du .
\end{align*}
In view of \eqref{C2reg} and Lemma \ref{lem:conseq_C2} $c)$, we obtain 
\begin{align}\label{eq:boundD2}
|D_t^1(T_0,T_1-1)| &\leq C_1 t \left( \frac{|W_{T_0}-W_{T_1-1}|}{(\alpha+1) \Delta^{\alpha+1}} + \sup_{r\in[T_0,T_1-1]} \frac{|W_r-W_{T_1-1}|}{(T_1-r)^{\frac{1}{2}+\epsilon}} \int_{T_0}^{T_1-1} (T_1-u)^{-\alpha-\frac{3}{2}+\epsilon} \dd u \right)\nonumber \\
&\leq C_1 \left(\frac{1}{\alpha+1} + \frac{1}{\alpha+\frac{1}{2}-\epsilon}\right) ~S^{\epsilon},
\end{align}
where
\begin{align*}
S^{\epsilon} := \sup_{s\in[T_0,T_1-1]} \frac{|W_s-W_{T_1-1}|}{(T_1-s)^{\frac{1}{2}+\epsilon}} .
\end{align*}
Setting $C_{1,\epsilon} = C_1 \left(\frac{1}{\alpha+1} + \frac{1}{\alpha+\frac{1}{2}-\epsilon}\right) $, we get
$$ \PP\left(\|D^1(T_0,T_1-1)\|_{\infty,[0,1]} \leq K_1 \right)\ge \PP\left(C_{1,\epsilon} S^{\epsilon} \leq K_1\right).$$ 
Let us prove that for a suitable $K_1$, the right-hand side is larger than $\sqrt{1-\eta}$.

To do so, we first deduce from Markov inequality that for all $s\geq 0$,
\begin{equation}\label{eq:def_calS}
\PP(C_{1,\epsilon} S^{\epsilon} \leq s) \geq 1- \frac{\EE e^{C_{1,\epsilon} S^{\epsilon}}}{e^{s }} ,
\end{equation}
But, 
\begin{align*}
\EE\left[ e^{C_{1,\epsilon} S^{\epsilon}}\right] =   \mathcal{E}(T_0-1,\Delta) ,
\end{align*}
where for any deterministic $\tau,\Delta\ge 1$,  
\begin{align*}
\mathcal{E}(\tau,\Delta) &:= \EE\left[ \exp\left\{C_{1,\epsilon} \sup_{s\in[1+\tau,\tau+\Delta]} \frac{|W_s - W_{\tau+\Delta}|}{(1+\tau+\Delta-s)^{\frac{1}{2}+\epsilon}} \right\} \right] \\
&= \EE\left[ \exp\left\{C_{1,\epsilon} \sup_{s\in[0,\Delta-1]} \frac{|W_s|}{(1+s)^{\frac{1}{2}+\epsilon}} \right\} \right] \leq \EE\left[ \exp\left\{C_{1,\epsilon} \|W\|_{\frac{1}{2}+\epsilon,\infty} \right\} \right]  .
\end{align*}
Hence 
\begin{align*}
\EE\left[ e^{C_{1,\epsilon} S^{\epsilon}}\right] \leq \EE\left[ e^{C_{1,\epsilon} \|W\|_{\frac{1}{2}+\epsilon,\infty}} \right] .
\end{align*}
Since we know from Fernique's theorem that $\EE \exp\{\lambda \|W\|_{\frac{1}{2}+\epsilon,\infty}\} < \infty$ for any $\lambda\in\R$ (see for instance \cite[Th. 4.1]{Ledoux96}), we deduce from the previous inequality and \eqref{eq:def_calS} that for 
\begin{equation*}
K_1 = \log \left( \frac{\EE \left[\exp\left\{C_{1,\epsilon}\|W\|_{\frac{1}{2}+\epsilon,\infty}\right\}\right]}{1-\sqrt{1-\eta}}\right)~,
\end{equation*}
the following inequality holds true:
\begin{equation*}
\PP\left(\|D^1(T_0,T_1-1)\|_{\infty,[0,1]} \leq K_1 \right) \geq \PP(C_{1,\epsilon} S^{\epsilon} \leq K_1) \geq 1-\frac{\EE e^{C_{1,\epsilon} S^{\epsilon}}}{e^{K_1}} \geq \sqrt{1-\eta} ~.
\end{equation*}
where $\eta$ does not depend on $T_0$ and $T_1$.

\smallskip

\textbf{$2^{\text{nd}}$ step.} Owing to the stationarity of the increments of the Wiener process,
$$\PP\left(\|D(T_1-1,T_1)\|_{\infty,[0,1]} \leq K_2 \right)=\PP\left(\|R\|_{\infty,[0,1]} \leq K_2\right)$$
where the process $R$ is defined in (\ref{eq:def_eq_R}). But by Lemma \ref{lem:boundR},  there exists $K_2$ such that
\begin{equation*}
\PP\left(\|R\|_{\infty,[0,1]} \leq K_2\right) \geq \sqrt{1-\eta}~.
\end{equation*}
This concludes the proof.

\end{proof}

\subsection{Exponential moments of $\tau_k$}

\begin{proposition}\label{prop:tailTau}
Let $\lambda>0$ and $r\in (0,1\wedge (2\delta)]$. With the previous notations, we have that 
\begin{align*}
\EE\left[ \exp\{\lambda \tau_k^r\} \right]  \leq e^{2\lambda k^{(\chi+1)r}} ~ \left(\EE\bigg[\exp\bigg\{ \lambda \|W\|_{\frac{1}{2}+\epsilon,\infty}^{\frac{r}{\delta}}\bigg\}\bigg] \right)^k.
\end{align*}
\end{proposition}

\begin{proof}
Recall that in the previous subsection, we defined $\tau_k = 1 + \tau_{k-1} + \Delta_k$ and $\Delta_k = k^\chi +  \left( S^{k-1,\epsilon}\right)^{\frac{1}{\delta}}$. Hence  $\tau_k = k + \sum_{j=1}^k \left( j^\chi +  (S^{j-1,\epsilon})^{\frac{1}{\delta}} \right) $ and 
\begin{align*}
\EE\left[ \exp\{\lambda \tau_k^r\} \right]  &\leq \exp\bigg\{\lambda \bigg(k^r + \big(\sum_{j=1}^k j^\chi\big)^r\bigg)\bigg\} \times \EE\bigg[\exp\bigg\{ \lambda\sum_{j=1}^k \left( S^{j-1,\epsilon}\right)^{\frac{r}{\delta}} \bigg\}\bigg] ,
\end{align*}
where we used the inequality $(x_1 + \dots + x_k)^r \leq x_1^r + \dots + x_k^r$, since $r\in(0,1]$.\\
It is clear that $\exp\bigg\{\lambda \bigg(k^r + \big(\sum_{j=1}^k j^\chi\big)^r\bigg)\bigg\} \leq \exp\{2\lambda k^{(\chi+1)r}\}$. 
Observe now that for any $k\geq 2$,
\begin{align*}
\EE\bigg[\exp\bigg\{ \lambda\sum_{j=1}^k \left( S^{j-1,\epsilon}\right)^{\frac{r}{\delta}} \bigg\}\bigg] &= \EE\bigg[\exp\bigg\{ \lambda\sum_{j=1}^{k-1} \left( S^{j-1,\epsilon}\right)^{\frac{r}{\delta}} \bigg\} \EE\bigg[\exp\bigg\{ \lambda\left( S^{k-1,\epsilon}\right)^{\frac{r}{\delta}} \bigg\} \mid \mathcal{F}_{1+\tau_{k-2}}\bigg] \bigg] \\
&= \EE\bigg[\exp\bigg\{ \lambda\sum_{j=1}^{k-1} \left( S^{j-1,\epsilon}\right)^{\frac{r}{\delta}} \bigg\} ~\widehat{\mathcal{E}}(\tau_{k-2},\Delta_{k-1}) \bigg],
\end{align*}
where for any deterministic $\tau>0$ and $\Delta>0$, 
\begin{align*}
\widehat{\mathcal{E}}(\tau,\Delta) &:= \EE\bigg[\exp\bigg\{ \lambda \left(\sup_{s\in(1+\tau,2+\tau+\Delta]} \frac{|W_s-W_{2+\tau+\Delta}|}{(3+\tau+\Delta-s)^{\frac{1}{2}+\epsilon}} \right)^{\frac{r}{\delta}}\bigg\}\bigg] \\
&= \EE\bigg[\exp\bigg\{ \lambda \left(\sup_{s\in(0,1+\Delta]} \frac{|W_s|}{(1+s)^{\frac{1}{2}+\epsilon}} \right)^{\frac{r}{\delta}}\bigg\}\bigg] \\
&\leq \EE\bigg[\exp\bigg\{ \lambda \|W\|_{\frac{1}{2}+\epsilon,\infty}^{\frac{r}{\delta}}\bigg\}\bigg] .
\end{align*}
Thus for any $k\geq 2$,
\begin{align*}
\EE\bigg[\exp\bigg\{ \lambda\sum_{j=1}^k \left( S^{j-1,\epsilon}\right)^{\frac{r}{\delta}} \bigg\}\bigg] \leq \EE\bigg[\exp\bigg\{ \lambda \|W\|_{\frac{1}{2}+\epsilon,\infty}^{\frac{r}{\delta}}\bigg\}\bigg] \times \EE\bigg[\exp\bigg\{ \lambda\sum_{j=1}^{k-1} \left( S^{j-1,\epsilon}\right)^{\frac{r}{\delta}} \bigg\}\bigg] 
\end{align*}
so that by an immediate induction, one gets
\begin{align*}
\EE\bigg[\exp\bigg\{ \lambda\sum_{j=1}^k \left( S^{j-1,\epsilon}\right)^{\frac{r}{\delta}} \bigg\}\bigg] \leq \left(\EE\bigg[\exp\bigg\{ \lambda \|W\|_{\frac{1}{2}+\epsilon,\infty}^{\frac{r}{\delta}}\bigg\}\bigg] \right)^k .
\end{align*}
\end{proof}

\section{Contraction between successive stopping times}\label{sec:contraction}

For  $K\in\R_+$, denote by $\mathcal{C}_0(K)$ the set of continuous processes $(d_t)_{t\in[0,1]}$ starting from $0$ and such that $\|d\|_{\infty,[0,1]}=\sup_{t\in[0,1]}  |d_t|\leq K$. 
By Proposition \ref{prop:contSDSSS}, under \eqref{C1} and \eqref{C2}, for any $x\in\R^d$ and any $d\in \mathcal{C}_0(K)$, we can $a.s.$ define $\{X^{x,d}_t, t\in[0,1]\}$ as the unique solution to
\begin{equation}\label{eq:def_Xd}
\forall t\in[0,1],\quad X^{x,d}_t = x + \int_0^t b(X^{x,d}_s) \dd s + d_t + {\sigma} Z_t,
\end{equation}
where 
\begin{equation*}
Z_t =\int_0^t \kerG(u-t) \dd W_u
\end{equation*}
is as in (\ref{eq:decompo-fbm-ter}) with $\tau=0$.

\begin{lemma} \label{lem:convexbis}
Assume that \eqref{C1} holds. Then, there exists $\bar{R}\ge R$ and $\bar{\kappa}\in(0,\kappa]$ such that for all $(x,y)\in(\ER^d)^2$, 
\begin{align*}
|x|\ge \bar{R} \Rightarrow \langle x-y, b(x)-b(y)\rangle \leq -\bar{\kappa} |x-y|^2.
\end{align*}
\end{lemma}
\begin{proof}
Owing to \eqref{C1}, it is enough to prove the result when $|y|\le R$. Let  $\bar{R}$ be a positive number strictly greater than $R$ and assume that $|x|\ge\bar{R}$. For a given $\beta\in (0,1]$, set $z_\beta=(1-\beta) x+\beta y$ {and
choose $\beta$ in such a way that $|z_\beta|= R$. For such a choice, $|x-z_\beta|\ge |x|-R$ and it follows that
$\beta\ge  \frac{|x|-R}{|x|+R}\ge \frac{\bar{R}-R}{R+\bar{R}}=:\beta(\bar{R}).$}
{Since $z_\beta$ and $x$ belong to $B(0,R)^c$, we can apply \eqref{C1} to obtain:} 
$$ \langle x-z_\beta, b(x)-b(z_\beta)\rangle \leq -\kappa |x-z_\beta|^2=-\kappa\beta^2 |x-y|^2$$
{and hence,
$$\langle x-y, b(x)-b(z_\beta)\rangle\le-\kappa\beta|x-y|^2\le -\kappa\beta(\bar{R})|x-y|^2.$$}
On the other hand, for any small positive $\varepsilon$
$$\langle x-y, b(z_\beta)-b(y)\rangle\le \frac{\varepsilon}{2}|x-y|^2+\frac{c_R}{2\varepsilon}$$
where $c_R:=\sup_{z,z'\in \bar{B}(0,R)} |b(z')-b(z)|^2<+\infty$. As a consequence, for all $|x|\ge \bar{R}$ and $y\in\ER^d$,
\begin{align*}
 \langle x-y, b(x)-b(y)\rangle&= \langle x-y, b(x)-b(z_\beta)\rangle+ \langle x-y, b(z_\beta)-b(y)\rangle\\
 &\le -\kappa \beta(\bar{R}) |x-y|^2+\frac{\varepsilon}{2}|x-y|^2+\frac{c_R}{2\varepsilon}\\
 &\le ( -\kappa\beta(\bar{R})+\frac{\varepsilon}{2}) |x-y|^2+\frac{c_R}{2\varepsilon}.
 \end{align*}
 Since $ \beta(\bar{R})\rightarrow1$ as $\bar{R}$ goes to infinity, we can fix $\bar{R}_0$ such that for any $\bar{R}\ge \bar{R}_0$, $\beta(\bar{R})\ge \frac{3}{4}$.
 Let $\bar{R}\ge \bar{R}_0$ and fix $\varepsilon=\kappa/2$. Then, set $\bar{R}$ large enough in such a way that  for any $|x|\ge \bar{R}$
 $$ |x-y|^2\ge (\bar{R}-R)^2\ge \frac{2 c_R}{\varepsilon \kappa}.$$
 Then, the result holds with $\bar{\kappa}=\kappa/4$.
 \end{proof}

Before stating the next lemma, which is crucial to prove the contraction, we recall the definition of the \emph{Cameron-Martin space} of $Z^{(i)}$. We refer to Chapter 8.4 and Appendix F in \cite{Janson} for a general account on the link between the Cameron-Martin space and its realisation as a reproducing kernel Hilbert space. For any $t\in[0,1]$, set the following function:
$$\forall s\in[0,1],\quad  K_t^{(i)}(s) = \int_0^1 g_i(u-t) g_i(u-s) \dd u. $$
There exists a Hilbert space $H(Z^{(i)})$ of functions on $[0,1]$ such that
$$ \forall s,t\in[0,1],\quad \langle K_t^{(i)},K_s^{(i)}\rangle_{H(Z^{(i)})} = K^{(i)}_t(s)= K^{(i)}_s(t) $$
and
\begin{align*}
H(Z^{(i)}) = \overline{\text{span}\{K_t^{(i)},~t\in[0,1]\}},
\end{align*}
where the completion is taken with respect to the norm of the inner product (which will be denoted by $\|\cdot\|_{H(Z^{(1)})}$). 
Recall from Remark \ref{rk:C2} that without loss of generality, we can assume that the support of $\kerG$ intersects $[-1,0]$. Moreover, in view of \eqref{C2supp} $Z$ has at least one non-degenerate component. Let us assume without loss of generality that the first component is non-degenerate.

\begin{lemma}\label{lem:support}
For any $\psi\in H(Z^{(1)})$ and any $\varepsilon>0$, 
$$\PE(\|Z^{(1)}-\psi\|_{\infty,[0,1]}\leq \varepsilon)>0.$$ 
Besides, there exist $t_0<t_1\in (0,1)$ such that: for any $a\in \R$, there is $\varphi \in H(Z^{(1)})$ which is continuous on $[t_0,1]$ and such that $\varphi(t_1) = a.$
\end{lemma}

\begin{proof}
As a consequence of the Cameron-Martin formula for Gaussian measures and the symmetry of the centred ball (see \cite[p.216]{Ledoux96} for the Cameron-Martin formula and \cite[Th. 3.1]{LiShao} for the inequality below), one gets
\begin{align*}
\PE(\|Z^{(1)}-\psi\|_{\infty,[0,1]}\leq \varepsilon) \geq e^{-\frac{1}{2}\|\psi\|_{H(Z^{(1)})}^2} \PE(\|Z^{(1)}\|_{\infty,[0,1]}\leq \varepsilon) .
\end{align*}
Hence we shall prove that the probability on the RHS of the previous inequality is positive. For this, consider the pseudo-metric induced by $Z^{(1)}$ on $[0,1]$, defined as $d_{Z^{(1)}}(s,t) = \left(\EE[(Z^{(1)}_s-Z^{(1)}_t)^2]\right)^{\frac{1}{2}}$, and its entropy number:
\begin{align*}
N([0,1],d_{Z^{(1)}},\varepsilon) = N(\varepsilon) := \inf\left\{n\in\N^*:~[0,1]\subseteq \bigcup_{j=1}^n B_j\right\},
\end{align*}
where the infimum runs over all $n$-uples of $d_{Z^{(1)}}$-balls of radius at most $\epsilon$.\\
Here we have
\begin{equation}\label{eq:boundVarZwithG}
d_{Z^{(1)}}(s,t) \leq \left(\EE[(G^{(1)}_s-G^{(1)}_t)^2]\right)^{\frac{1}{2}},
\end{equation}
and for $\mathfrak{h}$ as in Lemma \ref{lem:boundR}, we deduce from \eqref{eq:varG} that
\begin{align}\label{eq:boundEntropy}
N([0,1],d_{G^{(1)}},\varepsilon) \leq C \varepsilon^{-\frac{1}{\mathfrak{h}}}.
\end{align}

Note that if there is a map $F$ such that $N(\varepsilon)\leq F(\varepsilon)$ and if there exist $1<c_1\leq c_2<\infty$ such that for any $\varepsilon>0$, $c_1 F(\varepsilon) \leq F(\frac{\varepsilon}{2}) \leq c_2 F(\varepsilon)$, then from \cite{Talagrand93} (as formulated nicely in \cite[p.257]{Ledoux96}), it follows that
\begin{align*}
\PP\left(\sup_{t\in[0,1]} |Z^{(1)}_t|\leq \varepsilon\right)\geq  \exp\left\{-K F(\varepsilon)\right\}
\end{align*}
for some $K>0$. Hence according to Eq. \eqref{eq:boundEntropy}, we can choose $F(\varepsilon) = C \varepsilon^{-\frac{1}{\mathfrak{h}}}$.\\
Hence $\PP\left(\sup_{t\in[0,1]} |Z^{(1)}_t|\leq \varepsilon\right)>0$ for any $\varepsilon>0$ and thus $\PE(\|Z^{(1)}-\psi\|_{\infty,[0,1]}\leq \varepsilon)>0$.

We now turn to the second part of this proof. Let $t_1\in(0,1)$ be such that $\EE[(Z_{t_1}^{(1)})^2] = \int_0^1 g_1(u-t_1)^2 \dd u >0$. By continuity of the map $t\mapsto \int_0^1 g_1(u-t) g_1(u-t_1) \dd u$, there exist $t_0\in(0,t_1)$ and $\epsilon\in(0,t_0)$ such that 
\begin{align*}
\varphi_0 := \int_0^1 g_1(u-(t_0-\epsilon)) g_1(u-t_1) \dd u > 0 .
\end{align*}
Therefore the function $\varphi$ given by
\begin{align*}
\varphi(t) := \frac{a}{\varphi_0} \int_0^1 g_1(u-(t_0-\epsilon)) g_1(u-t) \dd u
\end{align*}
satisfies $\varphi(t_1)=a$ and is continuously differentiable on $[t_0,1]$ since $\varphi(t)= \frac{a}{\varphi_0} \int_0^{t_0-\epsilon} g_1(u-(t_0-\epsilon)) g_1(u-t) \dd u$ and $g_1$ is $\mathcal{C}^2$ on $[-1,\epsilon]$ according to \eqref{C2reg}.
\end{proof}

\begin{remark}
When $G$ is a fractional Brownian motion, $Z$ is the so-called Riemann-Liouville process. In that case, it is known that the Cameron-Martin space of $Z$ is equivalent to that of the fBm, which is dense in $C_0([0,1])$. Thus by a general result on Gaussian measures, the support of $\PP^Z$ is $C_0([0,1])$ (see for instance Theorem 3.6.1 in \cite{Bogachev}), which implies the conclusions of Lemma \ref{lem:support}.
\end{remark}

\begin{proposition}\label{prop:contraction}
Assume \eqref{C2}. Let $\bar{R}>0$ be defined by Lemma  \ref{lem:convexbis} and let $K>0$. 
\begin{itemize} 
\item[(i)] 
There exist some positive $\eta$ and  $\delta$ depending only on $K$ such that for any $x\in\ER^d$ and any $(d_t)_{t\in[0,1]}\in \mathcal{C}_0(K)$, some random times $0\le T_1< T_2\le 1$ exist such that the process $X^{x,d}$ defined by \eqref{eq:def_Xd} satisfies the following property with probability greater than $\eta$:
$$T_2-T_1\ge \delta\quad\textnormal{and}\quad (X^{x,d}_t)_{t\in[T_1,T_2]}\subset B(0,\bar{R})^c.$$
\item[(ii)] If \eqref{C1} holds, there exists $\varrho_1\in(0,1)$ such that for all $p>0$ and for all $x,y \in\ER^d$, 
$${\sup_{d\in \mathcal{C}_0(K)}}\ES[|X^{x,d}_1-X^{y,d}_1|^p]\le \varrho_1 |x-y|^p.$$
\end{itemize}
\end{proposition}
\begin{proof}
$(i)$ 
The proof is based on Lemma \ref{lem:support} and on the fact that $Z$ is almost surely $\alpha$-H\"older continuous for a given positive $\alpha\in(0,1)$ (for this last point, see Equation \eqref{eq:boundVarZwithG} and proceed as in Proposition \ref{lem:boundR}). According to the assumptions of the beginning of this section, we assume that $Z^{(1)}$ is the component of $Z$ with a non-degenerate support and let $\varphi$ and $t_0<t_1\in[0,1]$ be as in Lemma \ref{lem:support}.

The first idea is to build a deterministic path $\varphi:\ER\rightarrow\ER$ which, up to $\varepsilon$, guarantees to attain a contraction area. We emphasize that the path $\varphi$ is built carefully in order to avoid dependency on the parameters, and in particular on the initial condition $x$ and on $d$. This leads to very rough controls (the arguments could be refined in view of quantitative bounds): we calibrate a value $C_1$ such that for a small $\varepsilon$, for all process $(d_t)_{t\in[0,1]}$ such that $\|d\|_{\infty,[0,1]}\le K$,
$$|\varphi(t_1)|=C_1,\quad \|Z^{(1)}-\varphi\|_{\infty,[0,1]}\le \varepsilon \Longrightarrow \inf\{t\ge0, |X_t^{x,d}|\ge \bar{R}+2K+1\}=:T_1\le t_1.$$
To this end, let us remark that it is enough to prove the property when $|x|\le \bar{R}+2K+1$. In this case, assume that $T_1>t_1$. Then,  $(X_t^{x,d})_{[0,t_1]}\subset B(0,\bar{R}+2K+1)$ and hence, 
\begin{equation}\label{eq:minoxxdd}
|X_{t_1}^{x,d}|\ge |Z_{t_1}^{(1)}|-|x|-\|d\|_{\infty,[0,1]}-t_1\|b\|_{\infty,\bar{R}+2K+1}
\end{equation}
where for a given positive $r$, $\|b\|_{\infty,r}=\sup_{x\in  B(0,r)}|b(x)|$. But, if we set 
$$C_1=2(\bar{R}+2K+1)+K+t_1\|b\|_{\infty,\bar{R}+2K+1}+\varepsilon$$
we remark that the right-hand member of \eqref{eq:minoxxdd} is  greater than $\bar{R}+2K+1$ on the event 
  $$\Omega_1=\{\|Z^{(1)}-\varphi\|_{\infty,[0,1]}\le \varepsilon\}, $$
 which leads to a contradiction on $\Omega_1$. More precisely, if $\varphi$ is a (deterministic path) such that $|\varphi({t_1})|=C_1$, then for any $x\in\ER^d$ and $d$ such that $\|d\|_{\infty,[0,1]}\le K$, $T_1(\omega)\le t_1$ on $\Omega_1$. Let us now set 
 $$T_2=\inf\{t\ge T_1, |X^{x,d}_t-X^{x,d}_{T_1}|>2K+1\}.$$
If $T_2\ge 1$, the proof is achieved. Otherwise, we have on $\Omega_1$
\begin{align*}
2K+1=|X^{x,d}_{T_2}-X^{x,d}_{T_1}|&\le 2\|d\|_{\infty,[0,1]}+ \|\sigma\| |Z_{T_2}-Z_{T_1}|+(T_2-T_1)\|b\|_{\infty,\bar{R}+4K+2}.
\end{align*}
Let $\alpha\in(0,1)$. For a given $C_2>0$, let 
$\Omega_2=\{|Z_{T_2}-Z_{T_1}|\le C_2 |T_2-T_1|^\alpha\}.$
If $\omega\in \Omega_1\cap \Omega_2$, we thus have:
\begin{align*}
2K+1=|X^{x,d}_{T_2}-X^{x,d}_{T_1}|&\le 2\|d\|_{\infty,[0,1]}+(T_2-T_1)^\alpha(C_2+\|b\|_{\infty,\bar{R}+4K+2})
\end{align*}
and hence, for all $\omega\in \Omega_1\cap \Omega_2$,
$$(T_2-T_1)\ge \left(\frac{1}{C_2+\|b\|_{\infty,\bar{R}+4K+2}}\right)^{\frac{1}{\alpha}}=:\delta(C_2).$$
We can now conclude the proof. Let $\eta:=\frac{\PE(\Omega_1)}{2}$. By Lemma \ref{lem:support}, $\eta>0$. Let $\alpha>0$ such that $Z$ is $\alpha$-H\"older continuous. Then, there exists $C_2$ large enough such that $\PE(\Omega_2)\ge 1-{\eta}$. For this value, we set $\delta=\delta(C_2)$. Then, by construction, the announced statement is true on $\Omega_1\cap\Omega_2$ and we have
$${\overline{\eta}:=}\PE(\Omega_1\cap\Omega_2) \ge \PE(\Omega_1)+\PE(\Omega_2)-1\ge \eta.$$
This concludes the proof.

$(ii)$ Assume first that $p\ge2$. Let $F$ be the random ($a.s.$ ${\cal C}^1$-)function defined by $F(t)=e^{\frac{p}{2}\bar{\kappa} t}|X^{x,d}_t-X^{y,d}_t|^p$ where $\bar{\kappa}$ comes from Lemma  \ref{lem:convexbis}. We have:
$$F'(t)=e^{\frac{p}{2}\bar{\kappa} t}\left( \tfrac{p}{2}\bar{\kappa} |X^{x,d}_t-X^{y,d}_t|^p+p \langle X^{x,d}_t-X^{y,d}_t, b(X^{x,d}_t)-b(X^{y,d}_t)\rangle |X^{x,d}_t-X^{y,d}_t|^{p-2}\right).$$
By the first statement and Lemma  \ref{lem:convexbis}, we obtain that on $\Omega_\varphi{:=\Omega_1\cap \Omega_2}$, for every $t\in[T_1,T_2]$,
$F'(t)\le -\frac{p}{2}\bar{\kappa} F(t)$. Hence, 
$$  \mathbf{1}_{\Omega_\varphi} |X^{x,d}_{T_2}-X^{y,d}_{T_2}|^p\le \mathbf{1}_{\Omega_\varphi} \exp(-\tfrac{p}{2} \bar{\kappa} (T_2-T_1))|X^{x,d}_{T_1}-X^{y,d}_{T_1}|^p .$$
But since $\langle x-y,b(x)-b(y)\rangle \le 0$, the mapping $t\mapsto |X^{x,d}_{t}-X^{y,d}_{t}|^p$ is non-increasing and hence, 
$$\mathbf{1}_{\Omega_\varphi}|X^{x,d}_{1}-X^{y,d}_{1}|^p\le \mathbf{1}_{\Omega_\varphi} \exp(-\tfrac{p}{2} \bar{\kappa} (T_2-T_1))|X^{x,d}_{0}-X^{y,d}_{0}|^p {\leq }\mathbf{1}_{\Omega_\varphi}\exp(-\tfrac{p}{2} \bar{\kappa} {\delta})|x-y|^p .$$
Then, it follows that
\begin{align}
\EE\left[|X^{x,d}_{1}-X^{y,d}_{1}|^p\right]&\le \exp(-\tfrac{p}{2} \bar{\kappa} {\delta})|x-y|^p~ \PP(\Omega_\varphi) + \EE\left[\mathbf{1}_{\Omega_\varphi^c} |X^{x,d}_{1}-X^{y,d}_{1}|^p\right] \nonumber\\
&\leq \exp(-\tfrac{p}{2} \bar{\kappa} {\delta})|x-y|^p~ {\overline{\eta}} + |x-y|^p (1-{\overline{\eta}})\nonumber\\ 
&{\leq \left(1 -\bar\eta(1-\exp\{-\tfrac{p}{2} \bar{\kappa} \delta\})  \right)    |x-y|^p}\nonumber\\
&\leq \varrho_1  |x-y|^p ,\label{eq:jensen22}
\end{align}
where $\varrho_1 =\sup_{p\geq 2} \left(1-\bar\eta + \bar\eta \exp\{-\tfrac{p}{2} \bar{\kappa} \delta\} \right)= 1-\bar\eta + \bar\eta \exp(- \bar{\kappa} \delta)$ is in $(0,1)$ and is independent of $p$. This concludes the proof when $p\ge2$. When $p\in(0,2)$, we deduce from Jensen's inequality  that   Inequality \eqref{eq:jensen22} also holds in this case (and hence that the result is true for any $p>0$).
\end{proof}

We now assume that we are given a sequence satisfying (\ref{memorycondition}).

\begin{proposition}\label{prop:diffk_xy}
Assume that $\tau_0=0$ and that $(\tau_k)_{k\in\N}$ is a sequence of stopping times which satisfy (\ref{memorycondition}).  Then there exists $\varrho\in(0,1)$ such that (\ref{contraction}) holds true, i.e. for any $p>0$, $\forall k\in \N$,
\begin{equation*}
\EE[|X_{1+\tau_{k+1}}-Y_{1+\tau_{k+1}}|^p]\le \varrho\EE[|X_{1+\tau_{k}}-Y_{1+\tau_{k}}|^p]~.
\end{equation*}
\end{proposition}

\begin{proof}
Since  $\tau_{k+1}$ is ${\cal F}_{\tau_k}$-measurable and since $(W_{s+1+\tau_k}-W_{1+\tau_k})_{s\ge0}$ is independent of 
${\cal F}_{1+\tau_k}$, we have
\begin{equation*}
\PP\left(\|D(1+\tau_{k},\tau_{k+1})\|_{\infty,[0,1]} \leq K |{\cal F}_{1+\tau_k}\right)=\Pi(1+\tau_k,\tau_{k+1},K),
\end{equation*}
where for (deterministic) $T_0$, $T_1$ with $0\le T_0\le T_1$ and $K>0$,
$$\Pi(T_0,T_1,K)=\PP(\|D(T_0,T_1)\|_{\infty,[0,1]} \leq K).$$
Thanks to Proposition \ref{prop:recentPast} (applied with $\eta=1/2$), we deduce that a positive $K_r$ exists such that, 
$$ \pi_{k}:=\PP\left(\|D(1+\tau_{k},\tau_{k+1})\|_{\infty,[0,1]} \leq K_r  |{\cal F}_{1+\tau_k}\right) \geq \tfrac{1}{2}, \quad a.s.$$  
Set $\Omega_k^r = \{\|D(1+\tau_{k},\tau_{k+1})\|_{\infty,[0,1]} \leq K_r\}$, the corresponding event.  By Proposition \ref{prop:remotePast} (where we now write $K_R$ for the constant $C_{\epsilon,\delta_0}$), the whole past thus satisfies on $\Omega_k^r$:
\begin{equation}\label{eq:probacond22}
\|D(-\infty,\tau_{k+1})\|_{\infty,[0,1]} \leq \|D^{\Delta_{k+1}}(1+\tau_{k})\|_{\infty,[0,1]} + \|D(1+\tau_k,\tau_{k+1})\|_{\infty,[0,1]} \leq K_R+K_r ~.
\end{equation}
Thus, keeping in mind that $t\mapsto |X_{t}-Y_{t}|^p$ is $a.s.$ non-increasing, we get
\begin{align}
\EE[|X_{1+\tau_{k+1}}-&Y_{1+\tau_{k+1}}|^p]= \EE[\mathbf{1}_{\left(\Omega_k^r\right)^c}~|X_{1+\tau_{k+1}}-Y_{1+\tau_{k+1}}|^p]+\EE[\mathbf{1}_{\Omega_k^r}|X_{1+\tau_{k+1}}-Y_{1+\tau_{k+1}}|^p] \nonumber\\
&\leq \EE[\mathbf{1}_{\left(\Omega_k^r\right)^c}~|X_{1+\tau_{k}}-Y_{1+\tau_{k}}|^p] + \EE\left[\mathbf{1}_{\Omega_k^r}\EE\left[|X_{1+\tau_{k+1}}-Y_{1+\tau_{k+1}}|^p|\mathcal{F}_{\tau_{k+1}}\right]\right], \label{eq:probacond11}
\end{align}
where in the second line, we used that $\Omega_k^r$ belongs to $\mathcal{F}_{\tau_{k+1}}$. Then,
with the notations introduced above,
$$
\EE[\mathbf{1}_{\left(\Omega_k^r\right)^c}~|X_{1+\tau_{k}}-Y_{1+\tau_{k}}|^p]\le \ES[(1-\pi_k) |X_{1+\tau_{k}}-Y_{1+\tau_{k}}|^p].$$
%
For the second term of \eqref{eq:probacond11}, we intensively use  \eqref{eq:probacond22} and obtain that for $K=K_r+K_R$,
\begin{align*}
\mathbf{1}_{\Omega_k^r}\EE\left[|X_{1+\tau_{k+1}}-Y_{1+\tau_{k+1}}|^p|\mathcal{F}_{\tau_{k+1}}\right] &\leq \mathbf{1}_{\Omega_k^r}\sup_{d:\|d\|_{\infty,[0,1]}\leq K} \left.\EE\left[|X_1^{x,d}-X_1^{y,d}|^p  \right]\right|_{x = X_{\tau_{k+1}},y = Y_{\tau_{k+1}}} \\
&\leq \mathbf{1}_{\Omega_k^r} \varrho_1~ |X_{\tau_{k+1}} - Y_{\tau_{k+1}}|^p,
\end{align*}
in view of Proposition \ref{prop:contraction}$(ii)$. Hence, using again that $t\mapsto |X_{t}-Y_{t}|^p$ is $a.s.$ non-increasing, we deduce from \eqref{eq:probacond11} and from what precedes that 
\begin{align*}
\EE[|X_{1+\tau_{k+1}}-Y_{1+\tau_{k+1}}|^p] &\leq  \EE[(1-\pi_{k}) |X_{1+\tau_{k}}-Y_{1+\tau_{k}}|^p] + \varrho_1 \EE\left[\mathbf{1}_{\Omega_k^r} |X_{1+\tau_{k}}-Y_{1+\tau_{k}}|^p\right] \\
&\leq  \EE[(1-\pi_{k}) |X_{1+\tau_{k}}-Y_{1+\tau_{k}}|^p] + \varrho_1 \EE\left[\pi_k |X_{1+\tau_{k}} - Y_{1+\tau_{k}}|^p\right] \\
&\leq \EE[(1-(1-\rho_1)\pi_{k}) |X_{1+\tau_{k}}-Y_{1+\tau_{k}}|^p].
\end{align*}
Since $\pi_k$ is $a.s.$ greater than $1/2$, $1-(1- \varrho_1) \pi_{k} \in(0,\frac{(1+ \varrho_1)}{2})$.  The result follows with  $\varrho:=(1+ \varrho_1)/{2}$ (which belongs to $(0,1)$ since  $\varrho_1\in(0,1)$).
\end{proof}

\begin{remark}\label{rk:indepComp2}
Note that the assumption on the independence of the components of $G$ only appeared in Lemma \ref{lem:support}. 
Thus in order to extend Theorem \ref{th:maingen} to the case where the components of $G$ may be dependent, observe first that $Z^{(1)}$ is now a sum of $d$ independent processes with at least one of them having a non-degenerate support:
\begin{equation*}
Z^{(1)}_t = \sum_{j=1}^d \int_0^t g_{1j}(u-t)~dW_t^j.
\end{equation*}
It should be clear that the first part of Lemma \ref{lem:support} is unchanged given that $H(Z^{(1)})$ is identified. Thus we claim that $H(Z^{(1)})$ is now spanned by the functions $$K_t^{(1)}(s) = \sum_{j=1}^d \int_0^1 g_{1j}(u-t) g_{1j}(u-s)~du$$ and is still non-degenerate in view of Remark \ref{rk:C2}. Hence taking now $t_1$ such that $$\sum_{j=1}^d \int_0^1 g_{1j}(u-t_{1})^2~du>0,$$ the rest of the proof follows accordingly.
\end{remark}

\section{Proof of Theorems \ref{th:maingen} and \ref{th:maingen2}}\label{sec:proofTh3}

Recall that Theorem \ref{th:maingen} is a special case of Theorem \ref{th:maingen2} in the case of a fractional noise. Hence we present the proof of the latter, which is built as follows. In Subsection \ref{subsec:proof1}, we consider the $L^2$-control related to the parallel coupling of solutions to the SDE starting from $x$ and $y$ respectively. Finally, the proof of Theorem \ref{th:maingen2} is achieved in Subsection \ref{subsec:proof3} where we integrate our bounds with respect to the invariant distribution.

\subsection{Calibration of the parameters and proof of the $L^2$ bound}\label{subsec:proof1}
First, Proposition \ref{prop:tailTau} and the (exponential) Markov inequality yield for any $\lambda>0$:
$$\PE(\tau_{k+1}>t-1)\le e^{-\lambda (t-1)^r}  e^{2\lambda (k+1)^{(\chi+1)r}}  \left(\EE\bigg[\exp\bigg\{ \lambda \|W\|_{\frac{1}{2}+\epsilon,\infty}^{\frac{r}{\delta}}\bigg\}\bigg] \right)^{k+1}.$$
Thus, by Equation \eqref{eq:main_eq}, we deduce that for any $q\geq 1$ ($p=2q$) and any $t\geq 1$,
\begin{align*}
\EE[|X_t-Y_t|^q] \leq e^{-\frac{1}{2}\lambda (t-1)^r} \sum_{k\in\N} \EE[|X_{1+\tau_k}-Y_{1+\tau_k}|^{2q}]^{\frac{1}{2}} ~ e^{\lambda (k+1)^{(\chi+1)r}} ~ \left(\EE\bigg[\exp\bigg\{ \lambda \|W\|_{\frac{1}{2}+\epsilon,\infty}^{\frac{r}{\delta}}\bigg\}\bigg] \right)^{\frac{k+1}{2}}.
\end{align*}
In Proposition \ref{prop:diffk_xy}, we proved that \eqref{contraction} holds true for some $\varrho \in(0,1)$, thus
\begin{align*}
\EE[|X_t-Y_t|^q] \leq C_0 e^{-\frac{1}{2}\lambda (t-1)^r} \sum_{k\in\N} \varrho^{\frac{1}{2}k}  e^{\lambda (k+1)^{(\chi+1)r}} ~ \left(\EE\bigg[\exp\bigg\{ \lambda \|W\|_{\frac{1}{2}+\epsilon,\infty}^{\frac{r}{\delta}}\bigg\}\bigg] \right)^{\frac{k+1}{2}},
\end{align*}
where $C_0$ denotes $\EE[|X_0-Y_0|^{2q}]^{\frac{1}{2}}$.\\
Hence we aim at maximizing the rate $r$, while keeping the above sum finite. First, it is necessary that $(\chi+1)r\leq 1$. In view of condition  \eqref{cond:epsdeltachi}, there is also $\tfrac{1}{\chi+1}\leq \alpha_{\epsilon,\delta}$. Hence $r\leq \alpha_{\epsilon,\delta}$. On the other hand, for the exponential moment of $\lambda\|W\|_{\frac{1}{2}+\epsilon,\infty}^{\frac{r}{\delta}}$ to be finite, one must assume that $r\leq 2\delta$ (and $\lambda$ small enough in case $r=2\delta$). Since $\alpha_{\epsilon,\delta}$ decreases with $\delta$, $r$ will be maximised for $\delta_0$ such that $\alpha_{\epsilon,\delta_0} = 2\delta_0$. The solution is $\delta_0 = \tfrac{\alpha+\frac{1}{2}-\epsilon}{3}$ and the optimal rate is therefore
\begin{align*}
r_0 = \frac{2}{3}(\alpha+\frac{1}{2}-\epsilon),
\end{align*}
where $\epsilon$ is as small as desired. Note that $r_0 = 2\delta_0 = \alpha_{\epsilon,\delta_0}$, which implies that
\begin{align*}
\EE[|X_t-Y_t|^q] \leq C_0 e^{-\frac{1}{2}\lambda (t-1)^{r_0}} \sum_{k\in\N} \varrho^{\frac{1}{2}k}  e^{\lambda (k+1)} ~ \left(\EE\bigg[\exp\bigg\{ \lambda \|W\|_{\frac{1}{2}+\epsilon,\infty}^2\bigg\}\bigg] \right)^{\frac{k+1}{2}}.
\end{align*}
By Fernique's theorem, there exists $\lambda_0>0$ such that $\EE[\exp\{ \lambda \|W\|_{\frac{1}{2}+\epsilon,\infty}^2\}]<\infty$ if and only if $\lambda<\lambda_0$. Hence, denoting $F_{\epsilon,\lambda} = \tfrac{1}{2}\log\left(\EE[\exp\{ \lambda \|W\|_{\frac{1}{2}+\epsilon,\infty}^2\}] \right)$, the previous inequality now reads, 
\begin{align*}
\EE[|X_t-Y_t|^q] \leq C_0 e^{-\frac{1}{2}\lambda (t-1)^{r_0}} \varrho^{-\frac{1}{2}} \sum_{k\in\N} \exp\left\{-(k+1)\left(\tfrac{1}{2}|\log \varrho| -\lambda - F_{\epsilon,\lambda}  \right)  \right\}.
\end{align*}
Thus it is clear that there exists $\lambda_1\in (0,\lambda_0]$ such that for any $\lambda\in(0,\lambda_1)$, the above sum is finite.

\subsection{Proof of Theorem \ref{th:maingen2}}\label{subsec:proof3}

First, notice that the existence of the stationary law $\bar{\nu}$ of \eqref{eds} is given by Proposition \ref{prop:InvProbab}. Hence, one can now consider a random variable $Y_0 \sim \bar{\nu}$ and $Y$ the solution to \eqref{eds} started from $Y_0$. According to Proposition \ref{prop:InvProbab}, $Y_0$ has moments of any order. By a slight generalisation of \eqref{eq:main_eq} (that amounts to apply H\"older's inequality rather than Cauchy-Schwarz), one gets that for any $\upsilon>0$, any $q\geq 1$ and any random variable $X_0$ such that $\EE[|X_0|^{q + \upsilon}]<\infty$, the following holds: for any $\epsilon \in(0,\alpha+\tfrac{1}{2})$, there exists $C>0$ such that
\begin{align*}
\forall t\geq 0,\quad \EE|X_t-Y_t|^q \leq C \EE\left[|X_0-Y_0|^{q+\upsilon}\right]^{\frac{1}{q+\upsilon}} e^{-\frac{1}{C} t^\gamma} .
\end{align*}
where $\gamma = \tfrac{2}{3}(\alpha+\tfrac{1}{2}-\epsilon)$. 
In view of \eqref{eq:boundWass}, Equation \eqref{marginalresult} of Theorem \ref{th:maingen} now follows (for any noise satisfying \eqref{C2}). As for the functional version \eqref{functionalresult}, it is an easy consequence of the previous result and the fact that the mapping $t\mapsto \EE|X_t-Y_t|^q$ is non-increasing (see Remark \ref{rk:marginalToFunctional}). This concludes the proof of Theorem \ref{th:maingen2}.

\section{From Wasserstein to Total Variation Bounds}\label{sec:WassToTV}
In this part, the aim is to prove Theorems \ref{th:maingenTV} and \ref{th:maingenTV2}. As mentioned before, the idea of the proof is the following: for a given $t\ge0$, use first the rate of convergence in Wasserstein distance by letting the fBms being identical until time $t-1$. Then, attempt a coalescent coupling between times $t-1$ and $t$ and hope that the fact that the paths are very close (with high probability) leads in turn to a high probability of success (by success, we mean that $X_t=Y_t$). Such a strategy will work if one is able to have a precise estimation of the probability of success at time $1$ for two paths starting from two points $x$ and $y$. Let us remark that the non-Markov feature of the process leads to some specific difficulties. For instance, a strategy like the mirror coupling seems to be difficult to use here since such a coupling is only a way to ensure that the paths meet together in a finite time (which can be controlled). But  unfortunately, the price to pay to remain stuck seems to be too costly in this case. We thus follow the strategy initiated by Hairer \cite{hairer}, based on the addition of an adapted drift term and on the Girsanov theorem. However, we will see that such an approach works for the fractional Brownian motion for which the Volterra kernel has an explicit inverse but we will need to add \textit{ad hoc} assumptions in the general case.

\subsection{A first general property}\label{subsec:TVbound1}
The first step is independent of the Gaussian kernel. In this step, the idea is to identify a drift term which, added to the Gaussian noise of one of the components yields
a sticking at time $1$. 

\begin{proposition} \label{prop:coupling1} 
Assume that \eqref{C1} and \eqref{C2} hold. Then, there exists a {random} ${\cal C}^1$-function $\varphi_{\cal S}:\ER_+\rightarrow\ER^{d}$ adapted with respect to $\sigma(G_s,s\in(-\infty,t))$ such that the solution\footnote{By Proposition \ref{prop:contSDSSS}, existence and uniqueness holds $a.s.$} $(x_t,y_t)_{t\ge0}$ to the coupled SDE
$$
\begin{cases}
dx_t=b(x_t)\, dt+\sigma dG_t\\
dy_t=b(y_t)\, dt+\sigma \big( dG_t+\varphi_{\mathcal{S}}(t) \, dt \big) 
\end{cases}
$$
starting from $(x,y)$ satisfies $x_1=y_1$ ~$a.s.$ and such that
$$\|\varphi_{\cal S}\|_{\infty,[0,1]}\le {c |y-x|}\quad{a.s.},$$
where $c$ is a deterministic constant which does not depend on $(x,y)$. Furthermore, if {$b$ is Lipschitz continuous}, then for any $\beta\in (0,1)$, a positive constant $c$ exists such that
$$\|\varphi_{\cal S}'\|_{\infty,[0,1]}\le c|y-x|^{1-\beta} \quad a.s.$$ 
\end{proposition}

\begin{proof}
To build the function $(\varphi_{\cal S}(t))_{t\in[0,1]}$, one slightly adapts the proof of \cite[Lemma 5.8]{hairer}. More precisely, one sets $\rho(t)=y_t-x_t$ and remarks that
if $\varphi_{\cal S}$ is continuous, $\rho$ is a ${\cal C}^1$-function which is a solution to 
\begin{equation}\label{eq:ODE}
\rho'(t)=b(x_t+\rho(t))-b(x_t)+\sigma \varphi_{\cal S}(t).
\end{equation}
Let us notice that  $\rho$ is certainly a random function depending on $(x_t)_{t\in[0,1]}$ and thus on $(G_t)_{t\in[0,1]}$.  
Then, set $z(t)=|\rho(t)|^2$. Owing to Assumption \eqref{C1},
$$z'(t)\le 2 \langle\sigma \varphi_{\cal S}(t),\rho(t)\rangle.$$
Let $\beta\in(0,1)$. We can assume that \eqref{eq:ODE} is defined in such a way that
\begin{equation}\label{phisegalzero}
\varphi_{\cal S}(t)=- \varpi \sigma^{-1} \frac{\rho(t)}{{|\rho(t)|^\beta}},
\end{equation}
{for some $\varpi\in\R$,} with the convention $0/{|0|^\beta}=0$.
In this case, we obtain:
$$z'(t)\le  -2\varpi z(t)^{1-\frac{\beta}{2}} \quad \textnormal{on $[0,\tau_\rho:=\inf\{t\ge0,\rho(t)=0\}]$}$$
and $z'(t)=0$ if $t\ge \tau_\rho$. 
Thus,
$$
\forall t\in[0,1],\quad  |\rho(t)|\le \left( \left({|x-y|}^\beta-{\beta\varpi} t\right) \vee 0\right)^{{\frac{1}{\beta}}}
$$ 
and hence, if $\varpi= \frac{2 |x-y|^\beta}{\beta}$,  then 
\begin{equation}\label{ztegalzero}
 z(t)=0, \quad \forall t\in[1/2,1].
 \end{equation}
In particular, $z(1)=|y_1-x_1|^2=0.$
Furthermore, there exists $c$ independent of $x$ and $y$ such that
$$\|\varphi_{\cal S}\|_{\infty,[0,1]}\le {c |y-x|}\quad\textnormal{and}\quad \|\varphi_{\cal S}'\|_{\infty,[0,1]}\le c \varpi \|{\rho'}{|\rho|^{-\beta}}\|_{\infty,[0,1]}$$
But, if $b$ is Lipschitz continuous, a constant $c$ exists (which can change from line to line) such that 
$$\left|\rho'(t)\right|\le c(|\rho(t)|+\varpi|\rho(t)|^{1-\beta}),$$
and hence if $\beta\in(0,1)$,
$$ \|\varphi_{\cal S}'\|_{\infty,[0,1]}{\lesssim \||\rho|^{1-\beta} + \varpi |\rho|^{1-2\beta} \|_{\infty,[0,1]} } 
\lesssim {|y-x|}^{1-\beta}.$$
The result follows.
\end{proof}

Now, we need to control the corresponding underlying Wiener increments related to the moving-average representation \eqref{eq:def_noise}. More precisely, 
let $(x(t),\tilde{x}(t))_{t\ge0}$ be a couple of solutions to 
\begin{equation}\label{eq:ggtilde}
\begin{cases}
dx_t=b(x_t)\, dt+\sigma dG_t,\quad x_0=x,\\
d\tilde{x}_t=b(\tilde{x}_t)\, dt+\sigma  d\widetilde{G}_t,\quad \tilde{x}_0=y  ,
\end{cases}
\end{equation}
where $(G,\widetilde{G})$ is a couple of two-sided Gaussian processes with kernel ${\cal G}$ and underlying two-sided Wiener processes
$(W,\widetilde{W})$, as in  \eqref{eq:def_noise}. We also assume that 
 \begin{equation}\label{eq:wienerequal}
 (\widetilde{W}_t)_{t\le 0}=(W_t)_{t\le 0}\quad a.s.
 \end{equation}
 and hence that, $(\widetilde{G}_t)_{t\le0}=(G_t)_{t\le0}~ a.s.$ With these notations, one needs to answer the following question : if on a subset of $\Omega$,  $\widetilde{G}_t=G_t+ \int_0^t \varphi_{\cal S}(s) ds$, what must be the corresponding relationship 
between $\widetilde{W}$ and $W$ (on this same subset of $\Omega$) ? At this stage, we choose to separate the fractional and general cases:
 
 \subsection{The fractional case}\label{proof:theomaingentv} 
The proof of Theorem \ref{th:maingenTV} is achieved at the end of  this section and follows from the two next propositions.  
Here, we will denote the couple  $(G,\widetilde{G})$ introduced in \eqref{eq:ggtilde} by $(B^H,\widetilde{B}^H)$.

\begin{proposition}\label{prop:TV111}
Assume \eqref{eq:wienerequal}. Let $(x,y)\in \ER^d$ and let $(\varphi_{\cal S}(t))_{t\in[0,1]}$  be the adapted process defined in Proposition \ref{prop:coupling1} and assume that $b$
is Lipschitz continuous when $H>1/2$. 
\smallskip

\noindent (i) There exists a $\sigma(W_s,s\le .)$-adapted process $\Psi_{\cal S}$ such that 
$$ \widetilde{B}^H=B^H+ \int_0^. \varphi_{\cal S}(s) ds\quad\textnormal{as soon as}\quad \widetilde{W}=W+ \int_0^. \Psi_{\cal S}(s) ds,$$
and such that the following bound holds true:
$$
\int_0^1 |\Psi_{\cal S}(t)|^2 dt\le c {|y-x|^{\puisr}},\quad\textnormal{with}\quad \begin{cases}\puisr=2 &\textnormal{ if $H<1/2$}\\
\puisr=1 &\textnormal{ if $H>1/2$,}
\end{cases}
 $$
where $c$ is a positive  deterministic constant independent of $x$ and $y$.\smallskip

\noindent (ii) Let $(x_t,\tilde{x}_t)_{t\ge0}$ denote a solution to \eqref{eq:ggtilde}.  There exists a constant $C>0$ such that for any $x,y\in\ER^d$ such that $|x-y|\le 1,$
$$\| {\cal L}(x_1)-{\cal L}(\tilde{x}_1)\|_{TV}\le C{|y-x|}^{\frac{\puisr}{2}} $$
where $\puisr$ is defined in $(i)$. \smallskip

\noindent (iii)  Furthermore, there exists a constant $C>0$ such that for any $x,y\in\ER^d$ such that $|x-y|\le 1,$
$$\| {\cal L}(x_{1+.})-{\cal L}(\tilde{x}_{1+.})\|_{TV}\le C{|y-x|}^{\frac{\puisr}{2}}, $$
where for some path $(z(t))_{t\ge0}$ and a given $T>0$, $z_{T+.}=(z_{T+t})_{t\ge0}$ (and $\puisr$ is defined in $(i)$). 
\end{proposition}
\begin{remark}
When $H>1/2$, the result is still true for any $\puisr \in(0,1)$. Since it has no impact on the final exponent, we choose to state the result with $\puisr=1/2$. 
The third statement emphasizes the fact that one is able to keep the paths together  until infinity and that the cost is of the same order as the one for sticking the positions.
Let us remark that oppositely to \cite{hairer} where the strategy of proof is based on a series of attempts, there is only one attempt here. This has several consequences on the proof.
First, in the sticking part (corresponding to $(ii)$), a standard ``optimal coupling'' can be used since one does not need to worry about what happens when the coupling attempt fails. 
More precisely, it is not necessary to build a coupling strategy where  one controls the distance between the underlying Wiener processes in case of failure. Similarly, in $(iii)$ where the idea is to 
keep the paths together, the  strategy of \cite{hairer} was to try to get this property successively on a series on intervals (whose length increases exponentially) in order to preserve the possibility of trying again the attempt in case of failure. Here, the fact that there is only one attempt implies that the coupling strategy is built in such a way that at time $1$, there are two possibilities: staying together until infinity or failing.

\end{remark}
\begin{proof} (i) Once again, the proof follows the lines of \cite{hairer}. More precisely, by \eqref{djslkjdql}
 \begin{equation}\label{eq:invKernel}
 \Psi_{\cal S}(t)=c_H\frac{d}{dt}\left(\int_0^t (t-s)^{\frac{1}{2}-H} \varphi_{\cal S}(s) ds\right), \quad t\ge0,
 \end{equation}
{where $c_H = (\tfrac{1}{2}-H)\alpha_H$ is the same as in \eqref{eq:movavfbm}, for some $\alpha_H \in \R$.}
Thus, if  $H<1/2$,
$$ \Psi_{\cal S}(t)=c_H\int_0^t (t-s)^{-\frac{1}{2}-H} \varphi_{\cal S}(s) ds,$$
so that 
$$\|\Psi_{\cal S}\|_{\infty,[0,1]}\le c_H\|\varphi_{\cal S}\|_{\infty,[0,1]}\int_0^{1} (t-s)^{-\frac{1}{2}-H}ds\le C |y-x|,$$
where in the last line we used the controls established in Proposition \ref{prop:coupling1}. When $H>1/2$, one uses the last statement of Lemma 4.2 of \cite{hairer}:
$$\Psi_{\cal S}(t)=\frac{\alpha_H \varphi_{\cal S}(0)}{t^{H-\frac{1}{2}}}+\alpha_H\int_0^t \frac{\varphi_{\cal S}'(s)}{(t-s)^{H-\frac{1}{2}}}ds,$$
which leads to 
\begin{align*}
\int_0^1 |\Psi_{\cal S}(t)|^2 dt&\lesssim  |\varphi_{\cal S}(0)|^2+\|\varphi_{\cal S}'\|^2_{\infty,[0,1]}\int_0^1\left(\int_0^t (t-s)^{\frac{1}{2}-H} ds\right)^2 dt\\
&\lesssim \|\varphi_{\cal S}\|^2_{\infty,[0,1]}+ \|\varphi_{\cal S}'\|^2_{\infty,[0,1]}
\lesssim {|y-x|},
\end{align*}
by Proposition \ref{prop:coupling1} (applied {with $\beta=1/2$}).\smallskip

(ii) By construction, for any couple $(W,\widetilde{W})$ of Brownian motions on $[0,1]$, the corresponding couple of solutions satisfies:
$$\PE(x_1= \tilde{x}_1)\ge \PE\left(\widetilde{W}_t=W_t+ \int_0^t \Psi_{\cal S}(s) ds,\, t\in[0,1]\right).$$
As a consequence   
\begin{equation*}
\| {\cal L}(x_1)-{\cal L}(\tilde{x}_1)\|_{TV}\le 1- \sup_{(W,\widetilde{W})}\PE\left(\widetilde{W}_t=W_t+ \int_0^t \Psi_{\cal S}(s) ds,\, t\in[0,1]\right)=\frac{1}{2}\|\PE_W- \Upsilon^*\PE_W\|_{TV}
\end{equation*}
where $\Upsilon$ is defined by $\Upsilon(w)=w+\int_0^. \Psi_{\cal S}(s) ds$ and $\PE_W$ denotes the Wiener distribution on ${\cal C}([0,1],\ER^d)$. 

By Girsanov's Theorem, we know that $\Upsilon^*\PE_W$ is absolutely continuous with respect to
$\PE_W$ with density $D_1$ where $(D_t)_{t\ge0}$ is the true martingale (using $(i)$) defined by:
$$D_t(w)=\exp\left(\int_0^t \Psi_{\cal S}^w(s) dw(s)-\frac{1}{2}\int_0^t |\Psi_{\cal S}^w(s)|^2 ds\right) ~~t\in [0,1],\,\PE_W-a.s.$$
where we choose to write $\Psi_{\cal S}=\Psi_{\cal S}^w$ in order to keep in mind that $\Psi_{\cal S}$ is not deterministic.
Thus, by Pinsker inequality (see \cite{tsybakov}),
\begin{align*}
\|\PE_W- \Upsilon^*\PE_W\|_{TV}&\le \sqrt{\frac{1}{2}}\left(\int \log( D(w)^{-1})\PE_{W}(dw)\right)^{\frac{1}{2}}
&\le \frac{1}{2}\left({\int} \int_0^{1} |\Psi_{\cal S}^w(s)|^2 ds ~{\PE_{W}(dw)}\right)^{\frac{1}{2}}
&\le C|y-x|^{\frac{\puisr}{2}}
\end{align*}
by $(i)$. This concludes the proof.

\noindent $(iii)$ We prove that one can build a  coupling $(B^H,\widetilde{B}^H)$ such that the couple $(x,\tilde{x})$ of solutions to \eqref{eq:ggtilde} satisfies: $\PE(x_{1+.}\neq \tilde{x}_{1+.})\le C |y-x|^{\frac{\puisr}{2}}$. 
Denoting  by $(W,\widetilde{W})$ the underlying Wiener innovation processes, we assume that on $[0,1]$,
$$(\widetilde{B}_t^H)_{t\in[0,1)}= \left({B}_t+\int_0^t \varphi_{\cal S}(s) ds\right)_{t\in[0,1)}.$$
In other words, we suppose that the positions have stuck at time $1$. In order to keep the paths together after time $1$, we need that
\begin{equation}\label{sdskljdskaa}
\widetilde{B}^H_{t}-\widetilde{B}^H_{1}=B^H_{t}-B^H_1, \quad t>1.
\end{equation}
Then, let us remark that by \eqref{phisegalzero} and \eqref{ztegalzero}, $\varphi_{\cal S}(t)=0$ for all $t\in[1/2,1]$.  Thus,
owing to Lemma 4.2 of \cite{hairer} (Equation 4.11d), this implies that   \eqref{sdskljdskaa} holds true if 
$$ \forall t\ge 1,\quad \widetilde{W}_t= {W}_t+\int_0^t \Psi_{\cal S}(s) ds,$$
with 
$$ \forall t\ge1, \quad \Psi_{\cal S} (t)=c_H\int_0^{\frac{1}{2}} (t-s)^{-H-\frac{1}{2}} \varphi_{\cal S}(s) ds.$$
By construction, $(\Psi_{\cal S}(t))_{t\ge1}$ is a $\sigma(G_s,s\le 1)$-measurable function which satisfies $a.s.$:
$$\forall t\ge1,\quad |\Psi_{\cal S} (t)|\le C \|\varphi_{\cal S}\|_{\infty,[0,1]} \left(t-\tfrac{1}{2}\right)^{-H-{\frac{1}{2}}},$$
where $C$ is a deterministic constant independent of $x$ and $y$. By Proposition \ref{prop:coupling1}, we deduce that 
$$\int_1^{+\infty} |\Psi_{\cal S}(s)|^2 ds\le C|x-y|^{{2}} \int_{1}^{+\infty} \left(t-\tfrac{1}{2}\right)^{-2H-1} dt\le C |x-y|^2.$$
Combining with the first statement, one deduces that a universal constant $C$ exists such that for every $x,y$ such that $|x-y|\le 1$,
\begin{equation}\label{dqiziou}
\int_0^{+\infty} |\Psi_{\cal S}(s)|^2 ds\le C|x-y|^{{\puisr}}.
\end{equation}
By the same strategy {as in} $(ii)$, one deduces the result. Actually, by construction, 
$$\PE(x_{1+.}=\tilde{x}_{1+.})\ge \PE(\widetilde{W}_t= {W}_t+\int_0^t \Psi_{\cal S}(s) ds, t\in[0,\infty))$$
and hence, 
$$\|{\cal L}(x_{1+.})-{\cal L}(\tilde{x}_{1+.})\|_{TV}\le \frac{1}{2} \|\PE_W^{[0,\infty)}- \Upsilon^*\PE_W^{[0,\infty)}\|_{TV},$$
where $\PE_W^{[0,\infty)}$ denotes the Wiener distribution on ${\cal C}([0,\infty),\ER^d)$. Then, following the lines of \textit{(ii)} and using that $M_.=\int_0^. \Psi_{\cal S}^w(s) dW_s$ is a $L^2$-bounded and thus convergent martingale, one deduces from \eqref{dqiziou}
that 
$$\|\PE_W^{[0,\infty)}- \Upsilon^*\PE_W^{[0,\infty)}\|_{TV}\le  \frac{1}{2}\left({\int} \int_0^{+\infty} |\Psi_{\cal S}^w(s)|^2 ds ~\PE_{W}(dw)\right)^{\frac{1}{2}}\le  C|x-y|^{\frac{\puisr}{2}},$$
which yields the result.
\end{proof}

\begin{proposition}\label{prop:pastotv}
Let $t> 1$ and assume that $(X,Y)$ is a couple of solutions of the fractional SDE with underlying couple of Brownian motions $(W,\widetilde{W})$ satisfying almost surely $(\widetilde{W}_s)_{s\le t-1}=(W_s)_{s\le t-1}$. Assume that there exists $c_1>0$ and $\rho>0$ such that
\begin{equation}\label{eq:wassestirr}
\ES[|X_{t-1}-Y_{t-1}|]\lesssim \exp(-c_1 t^\rho).
\end{equation}
Then, there exists a constant $c_2>0$ and a coupling $( \widetilde{W}_s)_{s\in[t-1,+\infty]} \overset{(d)}{=} (W_s)_{s\in[t-1,+\infty]}$ such that
$$\PE((X_s)_{s\ge t}\neq  (Y_{s})_{s\ge t})\lesssim \exp(-c_2 t^\rho).$$
\end{proposition}

\begin{proof} Let $\varepsilon\in(0,1]$. 
Since we assume that $(\widetilde{W}_s)_{s\le t-1}=(W_s)_{s\le t-1}$, we deduce from Proposition \ref{prop:TV111} that the increments of $(W,\widetilde{W})$ can be built on $[t-1,\infty)$ in such a way that if $|X_{t-1}-Y_{t-1}|\le 1$, 
\begin{equation}\label{eq:tvestirrr}
\PE((X_s)_{s\ge t}\neq  (Y_{s})_{s\ge t}|X_{t-1},Y_{t-1})\le C|X_{t-1}-Y_{t-1}|^{\frac{r}{2}}.
\end{equation}
Then, 
$$\PE((X_s)_{s\ge t}\neq (Y_{s})_{s\ge t})\le \PE((X_s)_{s\ge t}\neq (Y_{s})_{s\ge t},~ |X_{t-1}-Y_{t-1}|\le \varepsilon)+\PE(|X_{t-1}-Y_{t-1}|> \varepsilon).$$

By the Markov inequality and \eqref{eq:wassestirr},
$$\PE(|X_{t-1}-Y_{t-1}|> \varepsilon)\le \frac{C}{\varepsilon} \exp(- c_1 t^{\rho}).$$
On the other hand, by  \eqref{eq:tvestirrr},
$$\PE((X_s)_{s\ge t}\neq (Y_{s})_{s\ge t},|X_{t-1}-Y_{t-1}|\le \varepsilon)\le\PE\left((X_s)_{s\ge t}\neq (Y_{s})_{s\ge t}\mid |X_{t-1}-Y_{t-1}|\le \varepsilon \right)\le C{\varepsilon}^{\frac{r}{2}}.$$
In order to optimise, we choose $\varepsilon$ in such a way that
 $$\frac{1}{\varepsilon} \exp(- c_1 t^{\rho})= \varepsilon^{\frac{r}{2}},$$
 $i.e.$
 $$\varepsilon= \exp\left(-\frac{2 c_1 }{(2+{r})} t^{\rho}\right).$$
 The result follows.
 \end{proof}

\noindent \textbf{Proof of Theorem \ref{th:maingenTV}:} Let us recall that the  first Wasserstein estimate of Theorem \ref{th:maingen} is obtained through a synchronous coupling. Hence, 
 \eqref{eq:wassestirr} holds with $\rho=\gamma$ (defined in Theorem \ref{th:maingen}). Theorem \ref{th:maingenTV} then is a direct consequence of Proposition \ref{prop:pastotv}.

\subsection{The general case}\label{subsec:TVbound3}

\subsubsection{Proof of Theorem \ref{th:maingenTV2}}
(i) Let us recall that Proposition \ref{prop:coupling1} does not depend on the noise process $(G_t)_{t\ge0}$. Thus, to prove the theorem, one only needs to extend 
Proposition \ref{prop:TV111}, $i.e.$ to control the underlying drift involved by the function $\varphi$ of Proposition \ref{prop:coupling1}.  As in Proposition \ref{prop:TV111}, we denote it by 
$\varphi_{\cal S}$ and we recall that $\varphi_{\cal S}$ is ${\cal C}^1$ on $[0,1]$. By Assumption  \eqref{C3}, $\widetilde{G}=G+\int_0^.\varphi_{\cal S}(s)ds$ if $\widetilde{W}=W+\int_0^.\Psi_{\cal S}(s)ds$, where $(W,\widetilde{W})$ denotes the underlying Wiener coupling and $\Psi_{\cal S}=(\Psi_{\cal S}^{(1)},\ldots, \Psi_{\cal S}^{(d)})$ is given by : for all $j\in\{1,\ldots,d\}$,
 $$\Psi_{\cal S}^{(j)}(t)=\frac{d}{dt}\left(\int_0^t h_j(s-t) \varphi_{\cal S}^{(j)}(s) ds\right), \quad t>0.$$
First, assume that $h_j$ satisfies \eqref{C3i} and let $\delta>0$. We have:
\begin{align*}
\frac{1}{\delta}\Big(\int_0^{t+\delta} h_j(s-t-\delta) \varphi_{\cal S}^{(j)}(s) ds-\int_0^t h_j(s-t) &\varphi_{\cal S}^{(j)}(s) ds\Big)=  \frac{1}{\delta}\int_t^{t+\delta}h_j(s-t-\delta)  \varphi_{\cal S}^{(j)}(s) ds\\
&+\int_0^t \frac{h_j(s-t-\delta)-h_j(s-t)}{\delta}\varphi_{\cal S}^{(j)}(s) ds.
\end{align*}
The fact that $\lim_{t\rightarrow0} h_j(t)=0$ implies that the first member in the right-hand side goes to $0$ as $\delta\rightarrow0$.  As a consequence,
 $$|\Psi_{\cal S}^{(j)}(t)|\le \| \varphi_{\cal S}^{(j)}\|_{\infty,[0,1]}  \limsup_{\delta\rightarrow0}\int_0^t \left|\frac{h_j(s-t-\delta)-h_j(s-t)}{\delta}\right|ds.$$
 But using that $h_j$ is ${\cal C}^1$ on $[-1,0)$, 
\begin{align*}
\int_0^t \left|\frac{h_j(s-t-\delta)-h_j(s-t)}{\delta}\right|ds&\le \frac{1}{\delta}\int_0^t\int_{t-s}^{t-s+\delta} |h_j'(-u)|du ds\\
&=\int_0^{t+\delta}  |h_j'(-u)| \frac{1}{\delta}\left(\int_{(t-u)\vee0}^{(t-u+\delta)\wedge t} ds\right) du\le \int_{0}^{t+\delta}  |h_j'(-u)| du.
\end{align*}
Then, by the integrability condition on $h'_j$ of Assumption \eqref{C3i} and  Proposition \ref{prop:coupling1}, one deduces that a constant $c$ exists such that for every $t\in(0,1]$,
 \begin{equation}\label{eq:genphis11}
 |\Psi_{\cal S}^{(j)}(t)|\le c \| \varphi_{\cal S}^{(j)}\|_{\infty,[0,1]}\le c|y-x|.
 \end{equation}

 \noindent Second,  consider the case where $h_j$ satisfies \eqref{C3ii} ({in particular, that $b$ is Lipschitz continuous}).  By Proposition \ref{prop:coupling1}, $ \varphi_{\cal S}$ is ${\cal C}^1$ on $[0,1]$. By an integration by parts, one obtains:
$$ \forall t>0, \quad \int_0^t h_j(s-t) \varphi_{\cal S}^{(j)}(s) ds={\cal I}_{h_j}(-t)  \varphi_{\cal S}^{(j)}(0)+\int_0^t {\cal I}_{h_j}(s-t)  (\varphi_{\cal S}^{(j)})'(s) ds,$$
where for $u\in(-\infty,0]$, ${\cal I}_{h_j}(u)=\int_u^0 h_j(v) dv$. 
Thus,
$$ \frac{d}{dt}\left( \int_0^t h_j(s-t) \varphi_{\cal S}^{(j)}(s) ds\right)={h_j}(-t)  \varphi_{\cal S}^{(j)}(0)+\frac{d}{dt}\left(\int_0^t {\cal I}_{h_j}(t-s)  (\varphi_{\cal S}^{(j)})'(s) ds\right).$$
Since $\lim_{t\rightarrow0^-} {\cal I}_{h_j}(t)=0$ and since $ h_j$ is locally integrable (by Assumption \eqref{C3ii}), 
 a similar argument as before shows that 
$$|\Psi_{\cal S}^{(j)}(t)| = \left|\frac{d}{dt}\left( \int_0^t h_j(s-t) \varphi_{\cal S}^{(j)}(s) ds\right)\right|\le  c\left(|{h_j}(-t)|  \|\varphi_{\cal S}^{(j)}\|_{\infty,[0,1]}+ \| (\varphi_{\cal S}^{(j)})'\|_{\infty,[0,1]}\right).$$
 Then, since $h_j$ belongs to $L^2([-1,0],\ER)$, one deduces from Proposition  \ref{prop:coupling1} 
 $$\int_0^1 (\Psi_{\cal S}^{(j)}(t))^2 dt\le c{|y-x|}.$$
 From what precedes and from \eqref{eq:genphis11}, one deduces that a constant $c$ exists such that
 $$\int_0^1 |\Psi_{\cal S}(t)|^2 dt\le c{|y-x|}.$$
 The sequel of the proof is exactly the same as the proof of Proposition \ref{proof:theomaingentv} $(ii)$.\smallskip
 
 \noindent (ii) Here, following carefully the lines of Proposition with \ref{proof:theomaingentv} $(iii)$, one remarks that
 the two paths remain stuck after time $1$ if on $(1,+\infty)$, $d\widetilde{W}_t=dW_t+\Psi_{\cal S} (t)dt$ with
$$ \forall i \in\{1,\ldots,d\},\forall t\ge1, \quad \Psi_{\cal S}^{(i)} (t)=\frac{d}{dt}\left(\int_0^{\frac{1}{2}} h_i(s-t) \varphi_{\cal S}(s) ds\right).$$
But,
$$\frac{d}{dt}\left(\int_0^{\frac{1}{2}} h_i(s-t) \varphi_{\cal S}(s) ds\right)=-\int_0^{\frac{1}{2}} h_i'(s-t) \varphi_{\cal S}(s) ds,$$
and hence, using Jensen inequality, 
$$\int_1^{+\infty} ( \Psi_{\cal S}^{(i)} (t))^2 dt\le \|\varphi_{\cal S}\|_{\infty,[0,1]}^2 \int_{0}^{\frac{1}{2}}\int_{1}^{+\infty} (h_i'(s-t))^2 dt ds.$$
By Proposition  \ref{prop:coupling1}, one deduces that,
$$\int_1^{+\infty} ( \Psi_{\cal S}^{(i)} (t))^2 dt\le c{|y-x|^2}\int_{\frac{1}{2}}^{+\infty} (h'_i(-u))^2 du.$$
But, under the additional assumption of  Theorem \ref{th:maingenTV2}$(ii)$,    $h'_i$ belongs to $L^2((-\infty,-1])$   (and thus in $L^2((-\infty,-1/2])$ since $h'_i$ is continuous on $(-\infty,0)$) and hence,
$$\int_1^{+\infty} ( \Psi_{\cal S}^{(i)} (t))^2 dt\le c{|y-x|^2}.$$
The sequel of the proof is exactly the same as the one of  Proposition \ref{proof:theomaingentv}$(iii)$.

\subsection{About the existence of $h_i$ in \eqref{C3}}\label{subsec:Laplace}

As mentioned before, the verification of Assumption \eqref{C3} seems to be a difficult problem that we choose not to address in this paper.
Nevertheless, in this section, we show that this problem (at least the existence of $h_i$) can be connected with the inversion of the Laplace transform of the kernel ${\cal G}$.
For the sake of simplicity, let us consider the one-dimensional case and  assume that Assumption  \eqref{C3} is fulfilled. Then, plugging \eqref{eq:defpsi} into \eqref{eq:linkPhiPsi} and  dropping the index $i$ for short, one gets
\begin{align*}
\varphi(t) = \frac{d}{dt} \left(\int_0^t g(s-t) \frac{d}{ds} \left(\int_0^s h(u-s) \varphi(u)~du\right) ~ds\right).
\end{align*}
Let us for instance treat the case of $h$ satisfying \eqref{C3i}. Then one gets
\begin{align*}
\varphi(t) = -\frac{d}{dt} \left(\int_0^t \varphi(u) \int_u^t g(s-t) h'(u-s)~ds~du \right) .
\end{align*}
{This equality  holds for any $\varphi\in\mathcal{C}^1(\ER_+;\R^d)$ and for any $t\ge0$ if and only if}
\begin{align*}
{\forall t\ge0,~\forall 0\le u\le t},\quad \int_{u}^t g(s-t) h'(u-s)~ds= -1,
\end{align*}
or equivalently that 
\begin{align*}
{\forall t\ge0},\quad \int_0^t g(v-t) h'(-v)~dv=-1.
\end{align*}
{For a function $f:\R\rightarrow \R$, denoting by $\Check{f}$ the function $f(-\cdot)$, the previous equality reads $\int_0^t \Check{g}(t-v) (\Check{h}(v))'~dv=1$. Denoting the Laplace transform by $\mathcal{L}_f(p) = \int_0^{+\infty} e^{-pt} f(t)~dt,~p>0$, it follows by applying it on both side of the previous equality that}
\begin{align*}
\forall p>0,\quad \mathcal{L}_{\Check{g}}(p) (p\mathcal{L}_{\Check{h}}(p)-h(0)) = \frac{1}{p}.
\end{align*}
Hence, it would suffice to find $h$ such that $\mathcal{L}_{\Check{h}}(p) = \frac{1}{p^2 \mathcal{L}_{\Check{g}}(p)}$. However it is generally a difficult matter to find, or even prove the existence, of the inverse Laplace transform. For instance, the Bromwich-Wagner formula provides a general criterion to invert the Laplace transform \cite[p.268]{mitrinovic}. To illustrate the limitations of this approach and the reason we do not develop this question further, we take the example of the fractional kernel $g(t) = t^{H-\frac{1}{2}}$. In that case, one has $\mathcal{L}_g(p) \approx p^{-H-\frac{1}{2}}$, and the map $\frac{1}{p^2 \mathcal{L}_g(p)}$ does not have the required properties to use the Bromwich-Wagner formula. However, we get formally that $\mathcal{L}^{-1}(p^{H-\frac{3}{2}})(t) = \frac{1}{\Gamma(\frac{3}{2}-H)}t^{\frac{1}{2}-H}$, which is the kernel appearing in \eqref{eq:invKernel}.

\vspace{1cm}

\appendix

\section{Invariant distribution of Gaussian driven SDEs}\label{App:InvDistrib}
In this section, one wishes to give some precisions about the definition and the existence of invariant distribution for general Gaussian driven SDEs (see \cite{cohen-panloup} for a similar but more probabilistic definition). As mentioned before, we use the construction of \cite{hairer} (related to fractional SDEs) by building a \textit{stochastic dynamical system} (SDS) over SDE \eqref{eds}. \smallskip

\noindent Denote by ${\cal C}^{\infty}_0(\ER_-)$, the set of ${\cal C}^\infty$-functions $w$ from $(-\infty,0]$ to $\ER$ such that $w(0)=0$ having compact support and set for given $\rho\in(0,1)$ {and $q\in\R$}
$$\|w\|_{\rho;{q}}=\sup_{t,s\in\ER_-}\frac{|w(t)-w(s)|}{|t-s|^{\frac{\rho}{2}}(1+|t|+|s|)^{\frac{1}{2}+{(q)}_+}}.$$
The application $w\mapsto \|w\|_{\rho;{q}}$ defines a norm on ${\cal C}^{\infty}_0(\ER_-)$ and one denotes by ${\cal H}_{\rho;{q}}$ the closure of  ${\cal C}^{\infty}_0(\ER_-)$ {in ${\cal C}_0(\ER_-)$} for the norm $\|\,.\,\|_{\rho;{q}}$. When $({q})_+$ is removed in the previous definition (or if $({q})_+=0$), we write simply $\mathcal{H}_\rho$ and $\|\cdot\|_\rho$ its norm, and it is proven in Lemma 3.5 of \cite{hairer} that ${\cal H}_{\rho}$ is a Polish space for any $\rho\in(0,1)$. 
The first step of the construction of the SDS consists in considering the Volterra-type operator related to the kernel $\kerG$. Following the lines of \cite{hairer}, we expect to be able, for each $i\in\{1,\ldots,d\}$, to give a ``regular'' construction of the moving-average operator  ${\cal D}_{g_i}$ related to \eqref{eq:def_noise}, where for a function $g:(-\infty,0]$ and a smooth function $w:\ER_{{-}}\mapsto \ER$ with compact support, the operator ${\cal D}_g$ is defined by 
  $${\forall t\in\R_-,\quad} {\cal D}_{g} w(t)=\int_{-\infty}^0 g(s) (w'(s+t)-w'(s)) ds, \quad w\in{\cal C}^{\infty}_0(\ER_-) .$$
 
{The next proposition gives the continuity of the operator $\mathcal{D}_g$, which will be important later for the construction of the stochastic dynamical system and will ensure the Feller property of its transition kernel.}
\begin{prop}\label{prop:contDg}
Assume that $g$ is a one-dimensional kernel satisfying \eqref{C2} and let $\rho\in(2\zeta\vee 0,1{\vee (1+2\alpha)})$ and $\tilde{\rho} = 2\wedge(\rho-2\zeta)$. Then the linear operator ${\cal D}_g$ is bounded (continuous)  from ${\cal H}_{\rho}$ to ${\cal H}_{\tilde{\rho};\zeta-\alpha}$.
\end{prop}
\begin{proof}
Our proof closely follows the one from \cite[Lemma 3.6]{hairer}, the difference lying in the use of assumption \eqref{C2} on the general kernel $g$. {Note that by Assumption \eqref{C2}, the interval $(2\zeta\vee 0,1{\vee (1+2\alpha)})$ is not empty.} 
We have to prove that  $\mathcal{D}_g$ is bounded, $i.e.$ that for any $w\in \mathcal{C}_0^\infty(\R_-,\R)$, $\|\mathcal{D}_g w\|_{\tilde{\rho};\zeta-\alpha} \leq C\|w\|_{\rho}$. Without loss of generality, let $0\geq t>s$ and set $h=t-s$. Assume first that $h\in(0,1]$.
\begin{align*}
\mathcal{D}_g w(t) - \mathcal{D}_g w(s) &= \int_{-\infty}^0 g(u) \left(w'(u+t) - w'(u+s)\right)~du \\
&= \int_{-\infty}^{s-h} \left\{g(u-t)-g(u-s)\right\} ~dw(u) - \int_{s-h}^s g(u-s)~dw(u) \\
&\quad+ \int_{s-h}^{t} g(u-t) ~dw(u) .
\end{align*}
Having in mind Lemma \ref{lem:conseq_C2}, it is clear that for $w\in\mathcal{H}_\rho$, $\rho>\zeta$, one has $g(-u) (w(u)-w(0)) \rightarrow 0$ as $u\rightarrow 0^-$. Thus one can integrate-by-parts each terms in the previous equation to get: 
\begin{align*}
\mathcal{D}_g w(t) - \mathcal{D}_g w(s) &= {-}\int_{-\infty}^{s-h} \left\{g'(u-t)-g'(u-s)\right\} (w(u)-w(s))~du\\
&\quad {+} \int_{s-h}^s g'(u-s) (w(u)-w(s)) ~du \\
&\quad {-} \int_{s-h}^{t} g'(u-t) (w(u)-w(t))~du + g(-2h)(w(t)-w(s)) \\
&=: T_1+T_2+T_3+T_4.
\end{align*}
Since $w\in \mathcal{H}_\rho$ and $g''$ satisfies \eqref{C2reg} and \eqref{C2Holder}, 
\begin{align*}
|T_1| &\leq C\|w\|_\rho h\int_{-\infty}^{s-1} (s-u)^{-\alpha-2} (s-u)^{\frac{\rho}{2}} \left(1+|u|+|s|\right)^{\frac{1}{2}}~du \\
&\quad\quad+ \|w\|_\rho h\int_{s-1}^{s-h} C(1+(s-u)^{-\zeta-2}) (s-u)^{\frac{\rho}{2}} \left(1+|u|+|s|\right)^{\frac{1}{2}}~du
\end{align*}
{where the assumptions on $\alpha$ and $\rho$ ensure that the first integral is finite,} and one can then check that this yields 
\begin{align*}
|T_1| \leq C\|w\|_\rho (1+|s|+|t|)^{\frac{1}{2}} \left(h+ h^{\frac{\rho-2\zeta}{2}}\right)\leq C\|w\|_\rho (1+|s|+|t|)^{\frac{1}{2}+(\zeta-\alpha)_+} h^{1\wedge\frac{\rho-2\zeta}{2}}.
\end{align*}
For $T_2$ we have, using Lemma \ref{lem:conseq_C2} d),
\begin{align*}
|T_2|&\leq \|w\|_\rho \int_{-h}^0 |g'(u)| (-u)^{\frac{\rho}{2}}(1+|u+s|+|s|)^{\frac{1}{2}}~du \\
&\leq C\|w\|_\rho (1+|s|+|t|)^{\frac{1}{2}}  \int_{-h}^0 (1+(-u)^{-\zeta-1}) (-u)^{\frac{\rho}{2}}~du \\
&\leq C\|w\|_\rho (1+|s|+|t|)^{\frac{1}{2}} (h^{1+\frac{\rho}{2}}+h^{\frac{\rho-2\zeta}{2}}) \\
&\leq C\|w\|_\rho (1+|s|+|t|)^{\frac{1}{2}+(\zeta-\alpha)_+} h^{\frac{\rho}{2}+ (1\wedge (-\zeta))} .
\end{align*}
The same bound is derived for $T_3$ and for $T_4$, we derive similarly $|T_4|\leq C (1+|s|+|t|)^{\frac{1}{2}+(\zeta-\alpha)_+} h^{\frac{\rho}{2}+ (0\vee (-\zeta))}$. Thus we get
\begin{align*}
|\mathcal{D}_g w(t) - \mathcal{D}_g w(s)| \leq C\|w\|_\rho (1+|s|+|t|)^{\frac{1}{2}+(\zeta-\alpha)_+} h^{\frac{1}{2}\tilde{\rho}},
\end{align*}
which concludes the case $h\in(0,1]$.

Consider now the case $h>1$. We get 
\begin{align*}
|T_1|&\leq C\|w\|_\rho h\int_{-\infty}^{s-h} (s-u)^{-\alpha-2} (s-u)^{\frac{\rho}{2}} \left(1+|u|+|s|\right)^{\frac{1}{2}}~du\\
&\leq C\|w\|_\rho h^{\frac{\rho}{2}-\alpha}  \left(1+|t|+|s|\right)^{\frac{1}{2}}\\
&\leq C\|w\|_\rho 
\begin{cases}
h^{\frac{\rho}{2}-\zeta}  \left(1+|t|+|s|\right)^{\frac{1}{2}+(\zeta-\alpha)_+} & \text{ if } \zeta\geq -1,\\
h^{1} \left(1+|t|+|s|\right)^{\frac{1}{2}} &\text{ if } \zeta< -1,
\end{cases}
\end{align*}
using that since $h> 1$, $h^{\zeta-\alpha} \leq \left(1+|t|+|s|\right)^{(\zeta-\alpha)_+}$ and in case $\zeta<-1$, $\tilde{\rho}=2$ and one has $h^{\frac{\rho}{2}-\alpha}\leq h = h^{\tilde{\rho}/2}$.\\
For $T_2$, we now use Lemma \ref{lem:conseq_C2} c) to get:
\begin{align*}
|T_2|&\leq \|w\|_\rho \int_{-h}^{-1} |g'(u)| (-u)^{\frac{\rho}{2}}(1+|u+s|+|s|)^{\frac{1}{2}}~du+\|w\|_\rho \int_{-1}^{0} |g'(u)| (-u)^{\frac{\rho}{2}}(1+|u+s|+|s|)^{\frac{1}{2}}~du \\
&\leq C\|w\|_\rho \int_{-h}^{-1} (-u)^{-(\alpha+1)+\frac{\rho}{2}}(1+|u+s|+|s|)^{\frac{1}{2}}~du+C \|w\|_\rho (1+|s|+|t|)^{\frac{1}{2}} \\
&\leq C\|w\|_\rho (h^{\frac{\rho}{2}-\alpha}+1) (1+|s|+|t|)^{\frac{1}{2}} ,
\end{align*}
and similarly to $T_1$, we deduce that $|T_2|\leq C\|w\|_\rho h^{1\wedge(\frac{\rho}{2}-\zeta)}(1+|s|+|t|)^{\frac{1}{2}+(\zeta-\alpha)_+}$. One can proceed similarly to verify that the same inequality holds for $T_3$ and $T_4$, and the claim follows.
\end{proof}

We shall use the exact same \emph{stationary noise process} that was constructed in Lemma 3.10 of \cite{hairer}, namely $$(\mathcal{H}_\rho, (\mathcal{P}_t)_{t\geq 0}, \mathcal{W}, (\theta_t)_{t\geq 0}),$$
where $\mathcal{W}$ is the Wiener measure on $\mathcal{H}_\rho$ (which is in fact $\mathcal{H}_\rho^{\times d}$, by a slight abuse of notations), $(\mathcal{P}_t)_{t\geq 0}$ is the transition semigroup associated to $\mathcal{W}$ (for which $\mathcal{W}$ is the only invariant measure) and $(\theta_t)_{t\geq 0}$ is an appropriate shift operator (see \cite[p.722-723]{hairer} for precise definitions).\smallskip

\noindent The second step is to show some existence, uniqueness and regularity properties related to SDE \eqref{eds}. To this end, consider for any $T>0$ and for each $x\in\R^d$ and each ${\mathfrak{g}} \in {\mathcal{C}_0([0,T])}$, the solution $\Xi_{{T}}(x,\mathfrak{g})$ of the following ODE:
\begin{align*}
\Xi_{{T}}(x,\mathfrak{g})(t) = x+ \int_0^t b(\Xi_{{T}}(x,\mathfrak{g})(s)) ~ds + \sigma \mathfrak{g}(t), \quad {t\in[0,T].}
\end{align*}
We have the following property:
\begin{prop}\label{prop:contSDSSS} 
If $b$ satisfies \eqref{C1}, then $(x,\mathfrak{g})\mapsto \Xi_{{T}}(x,\mathfrak{g})$, from $\ER^d\times{\cal C}([0,T],\ER^d)$ to ${\cal C}([0,T],\ER^d)$, is a well-defined function.
Furthermore,  $\Xi_{{T}}$ is locally Lipschitz continuous on $\ER^d\times{\cal C}([0,T],\ER^d)$.
\end{prop}

\begin{proof} 
This result corresponds to Lemma 3.9 of \cite{hairer}. The only difference lies in the assumptions on the drift function which are slightly more general in this setting (more precisely, we do not make assumptions on the derivative of $b$). We thus provide several details. First, let $x\in\ER^d$ and $\mathfrak{g}\in {\cal C}([0,T],\ER^d)$. For a given $t_0>0$, let
$F$ be the application from ${\cal C}([0,t_0],\ER^d)$ to ${\cal C}([0,t_0],\ER^d)$  defined by  $F(y)(t)= x+ \int_0^t b(y(s)) ds+\sigma \mathfrak{g}(t)$, $t\in[0,t_0]$. Let $A_{r,x}:=\{y:~y(0)=x,\|y-x\|_{\infty,[0,t_0]}\le r\}$.
The fact that $b$ is locally Lipschitz continuous implies that {there exist $r_0>0$ and} a constant $C_{r_0,x}$ such that for any $r\in(0,r_0]$, {any $y\in A_{r,x}$ and} any $t\in[0,t_0]$, 
$$ |F(y)(t)-x|\le C_{r_0,x} t+ {|\sigma|}\|\mathfrak{g}\|_{\infty,[0,t]},$$
so that for a small enough $t_0$, the set $A_{r,x}$ is stable by the application $F$. Furthermore, it can be checked that for $t_0$ small enough, the application $F$ is also
contractive on $A_{r,x}$ so that by the Banach fixed-point Theorem,  existence and uniqueness classically hold for $\Xi_{T}(x,\mathfrak{g})$ on ${\cal C}([0,t_0],\ER^d)$.
But, owing to Lemma \ref{lem:contrvsOU} below,  there exists a constant $C_T$ depending only $T$ such that 
 \begin{equation}\label{eq:contuniformofsol}
 \sup_{t\in[0,t_0]}|\Xi_{T}(x,\mathfrak{g})(t))|\le C_T(1+  |x|+ \|\mathfrak{g}\|_{\infty,[0,T]})^{N},
 \end{equation}
{where $N$ was defined in \eqref{C1}.} Then, a maximality argument shows that $\Xi_{T}(x,\mathfrak{g})$ is well-defined on $[0,T]$.\smallskip

Let us now prove the local Lipschitz property. For  any positive $r_1$ and $r_2$, 
set $B=\{(x,{\mathfrak{g}}), |x|\le r_1, \|{\mathfrak{g}}\|_{\infty,[0,T]}\le r_2\}$.  
Using that the control of the solutions established in  \eqref{eq:contuniformofsol} is locally uniform in the variable $(x,\mathfrak{g})$ (and available for $t_0=T$),
one deduces that a constant $C$ exists such that for any $(x,\mathfrak{g})$ and $(y, \tilde{\mathfrak{g}})\,\in B$,
$$|b(\Xi_{{T}}(x,g)_t)-b(\Xi_{{T}}(y,\tilde{\mathfrak{g}})_t)|\le C |\Xi_{{T}}(x,\mathfrak{g})_t-\Xi_{{T}}(y,\tilde{\mathfrak{g}})_t|.$$
By a Gronwall argument, this implies that $(x,\mathfrak{g})\mapsto \Xi_{{T}}(x,\mathfrak{g})$ is Lipschitz continuous on $B$.
\end{proof}

\begin{lemma}\label{lem:contrvsOU}
Assume \eqref{C1}. Let $\mathfrak{g}\in{\cal C}([0,\infty),\ER^d)$. Let $(x(t))_{\ge0}$ and $(y(t))_{\ge0}$ satisfying $\forall t\ge0$,
\begin{equation*}
x(t)=x+\int_0^t b(x(s)) ds+\sigma \mathfrak{g}(t)\quad\textnormal{and}\quad
y(t)=x-\int_0^t y(s) ds+ \sigma \mathfrak{g}(t).
\end{equation*}
Then, the following controls hold true: for any $T\ge0$, there exists a constant $C$ such that for any $t\in[0,T]$, 
$$|x(t)-y(t)|^2\le C\int_0^t e^{{\kappa}(s-t)}(1+ |y(s)|^{2N}) ds\quad\textnormal{and}\quad |y(t)|\le (|x|+{|\sigma|}\|\mathfrak{g}\|_{\infty,[0,T]}) e^{C T}. $$
\end{lemma}

\begin{proof}
First, \eqref{C1} implies that a constant $\beta$ exists such that  for any $x,y\in\ER^d$,
 $$\langle b(x)-b(y),x-y\rangle\le \beta-{\kappa} |x-y|^2 $$
 and hence
 $$\langle b(x){+}y,x-y\rangle\le \beta-{\kappa} |x-y|^2+ \langle b(y){+}y, x-y\rangle\le \beta {- \frac{\kappa}{2}} |x-y|^2+ \frac{C}{\kappa} (1+|y|^{2N}),$$
 where $C$ denotes a positive constant.
 Then, let $h$ denote the function defined by $h(t)=e^{{\kappa} t} {|x(t)-y(t)|^2}.$ We have
 $$ h'(t)=e^{{\kappa} t} \left({\kappa}{|x(t)-y(t)|^2}+2\langle b(x(t)){+}y(t),x(t)-y(t)\rangle \right)\le e^{{\kappa} t}({2}\beta+C(1+|y(t)|^{2N})).$$
 The first statement follows. As concerns the second one, this is a direct consequence of the Gronwall lemma.
 \end{proof}

{ For a given $T\ge0$, let $R_T$ denote the shift operator from ${\cal C}((-\infty,0],\ER^d)$ to ${\cal C}([0,T],\ER^d)$ defined by: for every $t\in[0,T]$, $R_T u (t)= u(t-T)-u(-T)$.
 This operator is needed to achieve the increments of $G$ in the following (at least formal) sense: for a given $t_0\ge0$,
$$G_{t+t_0}-G_{t_0},t\in[0,T]~|~\left(\{W_{s+t_0+ T}-W_{t_0+T}\}_{s\le0}=w\right) =(R_{T}\mathcal{D}_g w(t))_{t\in[0,T]}. $$
 In view of what precedes, one can now realise the SDE through the mapping
 \begin{equation*}
\begin{split}
\xi: \R_+\times \R^d\times \mathcal{H}_\rho &\rightarrow \R^d\\
(t,x,w) &\mapsto \Xi_{{t}}(x,R_{t}\mathcal{D}_g w)(t) .
\end{split}
\end{equation*}
From the continuity of $\mathcal{D}_g$ (Proposition \ref{prop:contDg}), the continuity of the embedding $\mathcal{H}_{\tilde{\rho};\zeta-\alpha}\hookrightarrow \mathcal{C}(\R_-;\R^d)$, the continuity of $(t,w)\mapsto R_t w$ on $\R_+\times \mathcal{C}(\R_-;\R^d)$ and the continuity properties of $\Xi_T$ (Proposition \ref{prop:contSDSSS}), one deduces that for any $T>0$, $t\mapsto \Xi_T(x,R_{t}\mathcal{D}_g w)(t)$ is continuous on $[0,T]$ and that $(x,w)\mapsto \Xi_T(x,R_{\cdot}\mathcal{D}_g w)$ is continuous from $\R^d\times \mathcal{H}_\rho$ to $\mathcal{C}([0,T];\R^d)$. Hence $\xi$ is a SDS in the sense of \cite[Definition 2.7]{hairer}.

This embedding of the SDE into this SDS structure leads to the definition of an homogeneous Feller Markov transition (see \cite{hairer} for details) and thus to invariant distributions on $\ER^d\times  \mathcal{H}_\rho $ (related to
this transition). We have the following result: 
\begin{proposition}\label{prop:InvProbab}
Under \eqref{C1} and \eqref{C2}, the SDS $\xi$ has an invariant probability measure, denoted by $\nu$. Besides, its projection $\bar{\nu}$ on $\R^d$   has moments of any order $p\in\N$.
\end{proposition}
\begin{proof} 
By a classical Krylov-Bogolyubov argument (see $e.g.$ \cite[Lemma 2.20]{hairer} for a similar approach), it is enough to show that for any $p\ge2$ and for any (generalised) initial condition $\mu$ on $\ER^d\times \mathcal{H}_\rho$ such that $\int |x|^p\mu(dx,dw)<+\infty$, we have $\sup_{t\ge0} \ES[ |X_t^{\mu}|^p]<+\infty$ (where, with a slight abuse of notation, $X^\mu$ denotes the solution starting from $\mu$).
First, one proves that this property holds true for an Ornstein-Uhlenbeck process $Y$  solution to $dY_t=-Y_t dt+ \sigma dG_t$.
Owing to an integration by parts, one classically remarks that $a.s.$ for any $t\ge0$,
\begin{equation}\label{eq:repoubis}
Y^\mu_t= e^{-t}(Y_0^\mu+\sigma \int_0^t e^{s} dG_s).
\end{equation}
 By  Lemma \ref{lem:boundVarOU} below, $\sup_{t\ge0} \ES[ |\int_0^t e^{-(t-s)} dG_s|^2]<+\infty$ and the fact that
$(\int_0^t e^{-(t-s)} dG_s)_t$ is a Gaussian process classically implies that in fact, $\sup_{t\ge0} \ES[ |\int_0^t e^{-(t-s)} dG_s|^p]<+\infty$ for any $p\ge2$. Thus, $\sup_{t\ge0} \ES[ |Y_t^{\mu}|^p]<+\infty$.\\
Second, consider the general case. By Lemma \ref{lem:contrvsOU} and Jensen inequality, one can check that a constant $C$ exists such that for any $t\ge0$, 
$$\ES[ |X_t^{\mu}|^p]<C\left(\ES[ |Y_t^{\mu}|^p]+\int_0^t e^{\kappa (s-t)}(1+ \ES[|Y_s^{\mu}|^{pN}]) ds\right).$$
The result follows.
\end{proof}

\vspace{0.3cm}

Let us now introduce the operator $\mathcal{D}_g^*$, which is the dual of $\mathcal{D}_g$ in the sense that for any $T>0$ and any $\phi$ smooth enough and with support in $[0,T]$,
\begin{align*}
\int_\R \mathcal{D}_g^*\phi(s) ~d(\mathcal{R}_T W)_s = \int_\R \phi(s) ~d(\mathcal{R}_T\mathcal{D}_gW)_s .
\end{align*}
The class of functions $\phi$ for which this relation holds is precised in the next paragraph. Note in particular that due to the formula \eqref{eq:def_noise}, $\mathcal{R}_T\mathcal{D}_gW$ is simply another way of writing $G$ on $[0,T]$.

Given a one-dimensional kernel $g$, we will consider the class of compactly supported in $\R_+$, locally integrable functions $\phi$ such that
\begin{itemize}
\item $\forall s\in\R$, $\displaystyle\lim_{\epsilon \rightarrow 0} \int_\epsilon^{+\infty} \left(\phi(s)-\phi(s+u)\right) g'(-u)~ du$ exists;
\item $\displaystyle\int_\R \left(\int_0^{+\infty} \left(\phi(s)-\phi(s+u)\right) g'(-u)~ du\right)^2 ds <\infty$.
\end{itemize}
We denote by $\mathcal{L}_g^2$ this class of functions and denote by $\mathcal{D}^*_g$ the operator which acts on $\phi\in \mathcal{L}_g^2$ as follows: $$\mathcal{D}^*_g\phi(s) = \int_0^{+\infty} \left(\phi(s)-\phi(s+u)\right) g'(-u)~ du.$$

\begin{lemma}\label{lem:boundVarOU}
For any $t>0$, the function defined by $\phi_t(s):= \mathbf{1}_{[0,t]}(s) e^{s-t},~s\in\R$, belongs to $\mathcal{L}_g^2$.\\
Hence, if $G$ is the one-dimensional Gaussian noise with kernel $g$ constructed on the two-sided Wiener process $W$, we have
\begin{align}\label{eq:intGWiener}
\forall t\geq 0,\quad \int_\R \phi_t(s) ~dG_s = \int_\R \mathcal{D}^*_g\phi_t(s) ~dW_s .
\end{align}
Besides, 
\begin{align}\label{eq:unifBoundOU}
\sup_{t\in\R_+} \EE\left[\left(\int_\R \phi_t(s) ~dG_s \right)^2\right] <\infty.
\end{align}
\end{lemma}

\begin{proof}
To prove that $\displaystyle\lim_{\epsilon \rightarrow 0} \int_\epsilon^{+\infty} \left(\phi_t(s)-\phi_t(s+u)\right) g'(-u)~ du$ exists, we use the continuous differentiability of $\phi_t$ on $(0,t)$ and the fact that $\lim_{u\rightarrow 0^+} |u g'(-u)|\leq C (u+u^{-\zeta})$ (see Lemma \ref{lem:conseq_C2} d)) which is integrable. Thus $\phi_t$ belongs to the domain of $\mathcal{D}^*_g$ (as does any continuously differentiable function).

Next we prove that 
\begin{align}\label{eq:suptVar}
\sup_{t\in\R_+} \int_\R \mathcal{D}^*_g\phi_t(s)^2~ds<\infty
\end{align}
(hence in particular that $\mathcal{D}^*_g \phi_t \in L^2(\R)$ for any $t\geq 0$). We have that
\begin{align}\label{eq:decompVar}
\int_\R \left(\mathcal{D}^*_g \phi_t(s)\right)^2 ds &= \int_{\R_-} \left(\int_{\R_+} -\mathbf{1}_{[0,t]}(s+u) e^{s+u-t} g'(-u)~du\right)^2 ds \nonumber\\
&\hspace{1cm} + \int_{\R_+} \left( \mathbf{1}_{[0,t]}(s) e^{s-t} \int_0^{+\infty} (1-\mathbf{1}_{[0,t]}(s+u)e^{u}) g'(-u)~du\right)^2 ds \nonumber\\
&= \int_{\R_+} \left(\int_0^t e^{u-t} g'(-(u+s))~du\right)^2 ds \nonumber\\
&\hspace{1cm} + \int_{\R_+} \left( \mathbf{1}_{[0,t]}(s) e^{s-t} \int_s^{+\infty} (1-\mathbf{1}_{[0,t]}(u)e^{u-s}) g'(s-u)~du\right)^2 ds,
\end{align}
where in the second equality, we performed the changes of variables $u \mapsto u-s$ and $s\mapsto -s$ for the first term, and $u\mapsto u-s$ for the second. It is clear that the supremum over $t\in[0,1]$ of the first term in the right-hand side of \eqref{eq:decompVar} is finite. Thus we assume in the following that $t\geq 1$. This reads
\begin{align*}
\int_{\R_+} \left(\int_0^t e^{u-t} g'(-(u+s))~du\right)^2 ds &\leq \int_0^1 \left(\int_0^1 g'(-(u+s))~du + \int_1^t e^{u-t} g'(-(u+s))~du\right)^2 ds\\
&\quad + \int_1^{+\infty} \left(\int_0^t e^{u-t} g'(-(u+s))~du\right)^2 ds. 
\end{align*}
Using Lemma \ref{lem:conseq_C2} c), we get that
\begin{align*}
\int_{\R_+} \left(\int_0^t e^{u-t} g'(-(u+s))~du\right)^2 ds &\leq C + C\int_1^{+\infty} \left(\int_1^t e^{u-t} (u+s)^{-(\alpha+1)}~du\right)^2 ds. 
\end{align*}
We now check that $\sup_{t\in[1,\infty)} \int_1^{+\infty} \left(\int_1^t e^{u-t} (u+s)^{-(\alpha+1)}~du\right)^2 ds<\infty$. In the following, recall that $\alpha>-\tfrac{1}{2}$ and assume that $\alpha\neq 0$ (this case can be easily treated separately):
\begin{align*}
\int_1^{\infty} \left(\int_1^t e^{u-t} (u+s)^{-(\alpha+1)}~du\right)^2 ds &= \int_2^{\infty}\int_2^{\infty} e^{u_1+u_2-2t} \left(u_1 u_2\right)^{-(\alpha+1)} \\
&\hspace{2cm}\times \int_1^{\infty} \mathbf{1}_{[u_1-t,u_1-1]\cap[u_2-t,u_1-1]}(s) e^{-2s} ~ds ~du_1~du_2\\
&\leq \tfrac{1}{2} \int_2^{\infty}\int_2^{\infty} e^{u_1+u_2-2t} \left(u_1 u_2\right)^{-(\alpha+1)} \\
&\hspace{2cm}\times\mathbf{1}_{\{[u_1-t,u_1-1]\cap[u_2-t,u_1-1]\neq \emptyset\}} e^{-2\left((u_1-t)\vee (u_2-t) \vee 1\right)} ~du_1 du_2 \\
&\leq \int_2^{\infty} \left\{\int_{2\vee (u_2-t+1)}^{u_2} e^{u_1} u_1^{-(\alpha+1)} ~du_1\right\} e^{u_2-2t-2\left((u_2-t)\vee 1\right)} u_2^{-(\alpha+1)} ~du_2\\
&\leq \frac{1}{-\alpha} \int_2^{t+1} \left(u_2^{-\alpha}-2^{-\alpha}\right) u_2^{-(\alpha+1)} e^{2(u_2-t-1)} ~du_2\\
&\hspace{2.5cm} + \int_{t+1}^\infty \int_{u_2-t+1}^{u_2} e^{u_1} u_1^{-(\alpha+1)}~du_1~ u_2^{-(\alpha+1)} e^{-u_2}~du_2.
\end{align*}
Thus there exists $C>0$ independent of $t\in\R_+$ such that
\begin{align*}
\int_1^{\infty} \left(\int_1^t e^{u-t} (u+s)^{-(\alpha+1)}~du\right)^2 ds &\leq C + \int_{t+1}^\infty \left(e^{u_2}-e^{u_2-t+1}\right) \left(u_2(u_2-t+1)\right)^{-(\alpha+1)} ~du_2 \\
&\leq 2C .
\end{align*}

As for the second term in \eqref{eq:decompVar}, it reads
\begin{align*}
\int_0^t e^{2(s-t)} &\left(\int_s^t \left(1-e^{u-s}\right) g'(s-u)~du + \int_t^{\infty} g'(s-u)~du\right)^2 ~ds \\
&= \int_0^{t-1} e^{2(s-t)} \left(\int_s^{s+1} \left(1-e^{u-s}\right) g'(s-u)~du + \int_{s+1}^t \left(1-e^{u-s}\right) g'(s-u)~du + g(s-t)\right)^2 ~ds \\
&\quad+ \int_{t-1}^t e^{2(s-t)} \left(\int_s^t \left(1-e^{u-s}\right) g'(s-u)~du + g(s-t)\right)^2 ~ds .
\end{align*}
We recall the following facts: 
\begin{itemize}
\item there exists $C>0$ (independent of $t$ and $s$) such that $|\int_s^{s+1} \left(1-e^{u-s}\right) g'(s-u)~du|\leq C$ (in view of Lemma \ref{lem:conseq_C2} d));
\item $C_g := \sup_{s\in(-\infty,-1]}|g(s)|<\infty$ (as a consequence of \eqref{C2reg});
\item $\int_{-1}^0 g(s)^2~ds<\infty$ (see \eqref{eq:kerG}),
\end{itemize}
and deduce from them that (recall that $C$ can change from line to line)
\begin{align*}
\int_0^t e^{2(s-t)} &\left(\int_s^t \left(1-e^{u-s}\right) g'(s-u)~du + \int_t^{\infty} g'(s-u)~du\right)^2 ~ds \\
&\leq \int_0^{t-1} e^{2(s-t)} \left(C + \int_{s+1}^t \left(1-e^{u-s}\right) g'(s-u)~du + C_g\right)^2 ~ds + \int_{t-1}^t e^{2(s-t)} \left(C + g(s-t)\right)^2 ~ds \\
&\leq C \int_0^{t} e^{2(s-t)} ~ds + 2 \int_0^{t-1} e^{2(s-t)} \left(\int_{s+1}^t \left(1-e^{u-s}\right) g'(s-u)~du \right)^2~ds + 2\int_{t-1}^t g(s-t)^2~ds \\
&\leq C +  2 \int_0^{t-1} e^{2(s-t)} \left(\int_{s+1}^t \left(1-e^{u-s}\right) g'(s-u)~du \right)^2~ds .
\end{align*}
Thus we focus on the remaining term, and using \eqref{C2reg} we get:
\begin{align*}
\int_0^{t-1} e^{2(s-t)} \left(\int_{s+1}^t \left(1-e^{u-s}\right) g'(s-u)~du \right)^2~ds &\leq C \int_0^{t-1} e^{2(s-t)} \left(\int_{1}^{t-s} e^{v} v^{-(\alpha+1)}~dv \right)^2~ds \\
&\leq C \int_0^{t-1} e^{2(s-t)} \left(e^{\frac{1+t-s}{2}}\int_{1}^{\frac{1+t-s}{2}}  v^{-(\alpha+1)}~dv \right)^2~ds\\
&\quad+C \int_0^{t-1} e^{2(s-t)} \left(\left(\frac{1+t-s}{2}\right)^{-(\alpha+1)}\int_{\frac{1+t-s}{2}}^{t-s} e^{v}~dv \right)^2~ds \\
&\leq C \int_0^{t-1} e^{s-t} ~ds + \int_0^{t-1} (1+t-s)^{-2(\alpha+1)} ~ds,
\end{align*}
which is bounded uniformly in $t$ since $\alpha>-\tfrac{1}{2}.$ Therefore the second term in the RHS of \eqref{eq:decompVar} is bounded for $t\in\R_+$ and so we have proven \eqref{eq:suptVar}.\\

Finally, one can verify that $(G_t)_{t\in\R_+} = \left\{\int_\R \mathbf{1}_{[0,t]}~dG_s\right\}_{t\in\R_+} \overset{(d)}{=} \left\{\int_\R \mathcal{D}^*_g\mathbf{1}_{[0,t]}(s)~dW_s\right\}_{t\in\R_+}$. Hence by approximation, \eqref{eq:intGWiener} is true. In particular, we see that \eqref{eq:unifBoundOU}, which is equivalent to \eqref{eq:suptVar}, holds.
\end{proof}

\section{Moving-average representation of $\R^d$-valued Gaussian processes with stationary increments}\label{App:PND}

In this section, we do not assume that the components of $G$ are independent. If that was the case, then Proposition \ref{prop:MovAvRep} below would be a straightforward generalisation of \cite[Theorem 4.2]{Cheridito}.

\subsection{Decomposition between purely nondeterministic and deterministic processes (Wold decomposition)}

In the following definition, $\overline{\textrm{sp}} A$ denotes the closure in $L^2(\Omega)$ of the vector space spanned by $A\subset L^2(\Omega)$.
\begin{definition}\label{def:pnd}
A process $(X_s)_{s\in\R}$ is said \emph{purely nondeterministic} if
\begin{align*}
\bigcap_{t\in\R} \overline{\textrm{sp}}\left\{X_s:~s\in(-\infty,t]\right\} = \{0\}
\end{align*}
and \emph{deterministic} if
\begin{align*}
\bigcap_{t\in\R} \overline{\textrm{sp}}\left\{X_s:~s\in(-\infty,t]\right\} = \overline{\textrm{sp}}\left\{X_s:~s\in(-\infty,\infty]\right\}
\end{align*}
\end{definition}
The representation of stochastic processes as a sum of a deterministic and purely nondeterministic process was an active field of research in the 50's and 60's, after the seminal work of Karhunen. We quote the following result which is well-suited to the framework of this paper.
\begin{proposition}[\cite{Cramer}, Theorem 3]
Let $(X_t)_{t\in\R}$ be an $\R^d$-valued stochastic process such that $\EE[|X_t|^2]<\infty,~\forall t\in\R$. Then $X$ has the following unique decomposition:
\begin{align*}
\forall t\in \R,\quad X_t \overset{(d)}{=} X_t^{\text{(det)}} + X_t^{\text{(pnd)}},
\end{align*}
where $X^{\text{(det)}}$ is a deterministic process, $X^{\text{(pnd)}}$ is a purely nondeterministic process, and $X^{\text{(det)}}$ and $X^{\text{(pnd)}}$ are orthogonal in the sense that $\forall (i,j)\in \{1,\dots,d\}^2,~\forall s,t\in\R$, $\EE[X^{(i,\text{det})}_s X^{(j,\text{det})}_t] = 0$.
\end{proposition}
In particular, such result can be used to describe the purely nondeterministic part in terms of an integral of a deterministic kernel against a process with orthogonal increments (see next subsection).

\subsection{Moving-average representation}

Most works focus on stationary processes. To extend to increment stationary processes, we use Masani's transform \cite{Masani}. This transform was already used for $\R$-valued processes in \cite{Cheridito} with the same purpose. Since we are not aware of the existence of this result for $\R^d$-valued processes, we recall Masani's transform and we outline and adapt the arguments of \cite{Cheridito}.

We say that that an $\R^d$-valued process $(X_t)_{t\in\R}$ is \emph{increment stationary} if  
\begin{align*}
\forall (i,j)\in&\{1,\dots,d\}^2,~\forall s,t,u,v,h,\\
& \EE\left[\left(X^{(i)}_{s+h}-X^{(i)}_{t+h}\right) \left(X^{(j)}_{u+h}-X^{(j)}_{v+h}\right)\right] = \EE\left[\left(X^{(i)}_{s}-X^{(i)}_{t}\right) \left(X^{(j)}_{u}-X^{(j)}_{v}\right)\right] .
\end{align*}
The definition of \emph{stationarity} is understood in a similar sense.

\begin{proposition}[Masani's transform \cite{Masani}]\label{prop:Masani}
Let $(X_t)_{t\in\R}$ be an $\R^d$-valued increment stationary process. If $X$ is continuous from $\R$ to $L^2(\Omega)$, then the process
\begin{align*}
\forall t\in\R,\quad Y_t := \int_{\R_+} e^{-u} \left(X_t-X_{t+u}\right) ~du
\end{align*}
is stationary and is the unique stationary process such that
\begin{align}\label{eq:invTransform}
\forall t\in\R,\quad X_t = X_0 + Y_t-Y_0 + \int_0^t Y_u~du .
\end{align}
\end{proposition}

\begin{proposition}\label{prop:MovAvRep}
If $G$ is a purely nondeterministic Gaussian process with stationary increments which satisfies 
\begin{equation*}
\lim_{t\rightarrow 0} \EE\left[|G_t|^2\right] = 0 ,
\end{equation*}
then it can be represented as 
\begin{equation*}
G_t \overset{(d)}{=} \int_{\R} \left\{\kerG(u-t) - \kerG(u) \right\} ~dW_u ,
\end{equation*}
where $W$ is an $\R^d$-valued standard Brownian motion and $\kerG$ is an $\Md$-valued function such that $\forall t>0,~ \kerG(t) =0$ and satisfying \eqref{eq:kerG}. 
\end{proposition}

\begin{proof}
This proof is a generalisation of the proof of \cite[Theorem 4.2]{Cheridito} which relies on the integral representation of stationary processes given in \cite{Gladyshev} {(note that a purely nondeterministic process is called ``regular'' in \cite{Gladyshev})}. Let $Y$ be the stationary process defined from $G$ as in Proposition \ref{prop:Masani}. {Like $G$, $Y$ is Gaussian and purely nondeterministic.} Then Theorem 2 of \cite{Gladyshev} implies that there exists an $\R^d$-valued process $W$ and a kernel $\tilde{\kerG}$ such that 
\begin{align*}
\tilde{\kerG}\in L^2(\R) \text{ and } \text{supp}~\tilde{\kerG} \subseteq \R_-,
\end{align*}
and
\begin{align*}
(Y_t)_{t\in\R} \overset{(d)}{=} \left\{\int_\R \tilde{\kerG}(u-t) ~dW_u\right\}_{t\in\R}.
\end{align*}
{Besides, this theorem states that $W$ satisfies $\EE[|dW^{(i)}_t|^2]=dt$ and ${\EE[(W^{(i)}_t-W^{(i)}_s)(W^{(j)}_{t'}-W^{(j)}_{s'})]=0}$ whenever $i\neq j$ or $[s,t]\cap [s',t'] = \emptyset$. From the construction of $W$ in the proof of \cite{Gladyshev}, it also appears that $W$ is Gaussian. In view of these properties, $W$ is a standard $\R^d$-valued Brownian motion.}\\
It remains to apply the inverse transform \eqref{eq:invTransform} to find that 
\begin{align*}
\forall t\in\R_-,\quad \kerG(t) = \tilde{\kerG}(t) + \int_{t}^0 \tilde{\kerG}(v)~dv,
\end{align*}
and $\kerG(t) = 0$ if $t>0$.
\end{proof}

\end{document}